\documentclass[leqno]{amsproc}
\usepackage{amsfonts,amssymb,amsmath,amsgen,amsthm}
\usepackage{hyperref}
\usepackage{color}
\newcommand{\msc}[2][2000]{%
  \let\@oldtitle\@title%
  \gdef\@title{\@oldtitle\footnotetext{#1 \emph{Mathematics subject
        classification.} #2}}%
}

\theoremstyle{plain}
\newtheorem{theorem}{Theorem} [section]
\newtheorem{definition}[theorem]{Definition}

\newtheorem{lemma}[theorem]{Lemma}

\newtheorem{proposition}[theorem]{Proposition}

\theoremstyle{remark}
\newtheorem{remark}[theorem]{Remark}

\def\R{{\mathbb R}}
\def\H{{\mathcal H}}
\def\Sch{{\mathcal S}}
\def\O{\mathcal O}
\def\F{\mathcal F}

\def\({\left(}
\def\){\right)}
\def\<{\left\langle}
\def\>{\right\rangle}
\def\le{\leqslant}
\def\ge{\geqslant}

\def\Eq#1#2{\mathop{\sim}\limits_{#1\rightarrow#2}}
\def\Tend#1#2{\mathop{\longrightarrow}\limits_{#1\rightarrow#2}}

\def\d{{\partial}}
\def\eps{\varepsilon}

\def\fum{{\widehat{u}_-}}
\def\um{{u_-}}
\def\u{{\mathbf u}}

\def\W{{W_-^{\rm mod}}}
\def\Scatt{{S^{\rm mod}}}
\def\cR{{\mathcal R}}

\DeclareMathOperator{\RE}{Re}

\numberwithin{equation}{section}

\begin{document}
\title[Long range scattering for cubic NLS]{Dynamics near the origin
  of the long range scattering for the 
  one-dimensional Schr\"odinger equation}    
\author[R. Carles]{R\'emi Carles}
\address{Univ Rennes, CNRS\\ IRMAR - UMR 6625\\ F-35000
  Rennes, France}
\email{Remi.Carles@math.cnrs.fr}

\begin{abstract}
  We consider the cubic Schr\"odinger equation on the line, for which
  the scattering theory requires modifications due to long range
  effects. We revisit the construction of the modified wave operator,
  and recall the construction of its inverse, in order to describe the
  asymptotic behavior of these operators near the origin. At leading
  order, these operators, whose definition includes a nonlinear
  modification in the phase compared to the linear dynamics, correspond to the identity. We compute
  explicitly the first corrector in the asymptotic expansion, and
  justify this expansion by error estimates. 
\end{abstract}

\thanks{This work was supported by Centre Henri Lebesgue,
program ANR-11-LABX-0020-0. A CC-BY public copyright license has
  been applied by the author to the present document and will be
  applied to all subsequent versions up to the Author Accepted
  Manuscript arising from this submission.}
\maketitle

\section{Introduction}
\label{sec:intro}
We consider the cubic Schr\"odinger equation on the line,
\begin{equation}
  \label{eq:NLS}
  i\d_t u +\frac{1}{2}\d_{x}^2u =\lambda |u|^2u,\quad x\in \R, 
\end{equation}
with $\lambda \in \R$. It is well known that in terms of scattering,
this equation corresponds to a borderline case, where long range
effects appear (see \cite{Barab}). The goal of this paper is to
analyze the modified scattering map near the origin, as well as the
modified wave operator and its inverse. We first recall a few aspects
of scattering theory for the nonlinear Schr\"odinger equation in the short
range case, then turn to the specificities of this long range setting.  

\subsection{Short range scattering}

Denote by $U(t)$ the  Schr\"odinger group,
\begin{equation*}
  U(t):= e^{i\frac{t}{2}\d_x^2}.
\end{equation*}
In view of the explicit formula for $U(t)$ as a convolution operator,
we have
\begin{equation}\label{eq:asym-lin}
  U(t)f(x)\Eq t {\pm\infty} e^{i\frac{x^2}{2t}} \frac{1}{(it)^{1/2}}\widehat
  f\(\frac{x}{t}\),
\end{equation}
where we normalize the Fourier transform as follows,
\begin{equation*}
  {\F} f(\xi)=\widehat
  f(\xi)=\frac{1}{\sqrt{2\pi}}\int_{\R}f(x)e^{-ix \xi} dx.
\end{equation*}
See Lemma~\ref{lem:R} for a more precise statement regarding
\eqref{eq:asym-lin}. 
Consider, for the
sake of comparison with \eqref{eq:NLS}, the case of a quintic,
defocusing nonlinearity
\begin{equation}
  \label{eq:NLSquintic}
  i\d_t u +\frac{1}{2}\d_{x}^2u = |u|^4 u,\quad x\in \R.
\end{equation}
The discussion for, e.g.,  the two-dimensional or the three-dimensional
cubic Schr\"odinger equation would be similar. Given 
\begin{equation}\label{eq:Sigma}
  u_-\in \Sigma:=\left\{f\in H^1(\R);\
    \|f\|_\Sigma:= \|f\|_{L^2}+\|\d_x f\|_{L^2} + \|x
    f\|_{L^2}<\infty\right\},
\end{equation}
there exists a unique $u_0\in \Sigma=H^1\cap \F(H^1)$ such that the
(unique, global) 
solution $u$ to \eqref{eq:NLSquintic} with $u_{\mid t=0}=u_0$
satisfies 
\begin{equation*}
  \|U(-t)u(t)-u_-\|_{\Sigma}\Tend t {-\infty}0.
\end{equation*}
We recall that $U(t)$ is unitary on $H^1$, but not on $\F(H^1)$, this
is why the quantity measured above is not $u(t)-U(t)u_-$. 
The map $u_-\mapsto u_0$ is classically referred to as \emph{wave
  operator} (see e.g. \cite{Ginibre}).

Conversely, given $u_0\in \Sigma$, there exists $u_+\in \Sigma$ such
that 
\begin{equation*}
  \|U(-t)u(t)-u_+\|_{\Sigma}\Tend t {+\infty}0.
\end{equation*}
The map $u_-\mapsto u_+$ is called the \emph{scattering operator}. It
can be defined for other (defocusing) nonlinearities, but strictly
supercubic, in view of \cite{Barab}, where it is proved that,
typically for \eqref{eq:NLS}, it is not possible to compare the
nonlinear dynamics with the linear one for large time (see also \cite{Ginibre}). We will see in the next
subsection that in the cubic case \eqref{eq:NLS}, nontrivial long range
effects must be taken into account, and that these effects are explicit.
\smallbreak

In general, rather little is known regarding properties of the
scattering map $S$.  For instance, it is proven in \cite{CaGa09} that
for smooth, power-like nonlinearities such that short range scattering
is known, the wave and scattering operators are analytic. The formula
for the associated Taylor series is given, and the formula differs
whether the expansion is considered at the origin or at a nontrivial
state. Such asymptotic expansions (not necessarily in the analytic
case) have proven useful in the context of inverse problems, see
e.g. \cite{MR4576319,KiMuVi-p} and references therein.

\subsection{Long range case: modified scattering}
\label{sec:mod}

The picture is different in the case of \eqref{eq:NLS}: the nonlinear
dynamics cannot be compared to the linear one, unless one considers
the trivial case $u\equiv 0$ (\cite{Barab}), so a modified scattering
theory has been developed in order to describe the large time behavior
of $u$, sharing some similarities with the linear case, presented in
e.g. \cite{DG}. 
\smallbreak

Given asymptotic states $u_-$ and $u_+$, we introduce the long range phase
corrections (different whether $t\to -\infty$ of $t\to +\infty$),
\begin{equation}
  \label{eq:defS}
  S_\pm(t,x):=\mp\lambda \left| \widehat{u}_\pm\left( \frac{x}{t}
  \right)\right|^2 \log |t|.
\end{equation}
Loosely speaking, the existence of modified wave operator reads as
follows: given $u_-$ (sufficiently small),  there exists a 
solution $u$ to \eqref{eq:NLS} such that  
\begin{equation*}
  u(t,x)\Eq t {-\infty} e^{iS_-(t,x)}U(t)u_-(x)\Eq t {-\infty}
  e^{iS_-(t,x)+i\frac{x^2}{2t}} \frac{1}{(it)^{1/2}}\widehat 
  u_-\(\frac{x}{t}\),
\end{equation*}
where the second approximation stems from \eqref{eq:asym-lin}. Note
that the phase modification $S_-$ is by no means negligible 
as $t\to -\infty$: it accounts for long range effects, as established
initially in \cite{Ozawa91}. We denote by $u_{\mid t=0} =\W (u_-)$ the
modified wave operator. 
\smallbreak

The modified asymptotic completeness is similar: given $u_0$
(sufficiently small), there exists $u_+$ such that the solution $u$ to
\eqref{eq:NLS} with $u_{\mid t=0}=u_0$ satisfies
\begin{equation*}
  u(t)\Eq t {+\infty} e^{iS_+(t)}U(t)u_+.
\end{equation*}
Such a result was proven initially in \cite{HN98}. 
The modified scattering map is given by $\Scatt (u_-)=u_+$. At this stage,
we have not addressed the function spaces in which the above
asymptotics have been proven. 
\smallbreak

In \cite{Ozawa91}, the existence of modified wave operators was
established for $u_-\in \F(H^2)$, with $\|\widehat u_-\|_{L^\infty}$
sufficiently small, and the solution $u$ to \eqref{eq:NLS} has an
$L^2$ regularity, $u_0\in L^2(\R)$. We emphasize, as the notation will
be used many times, that the space $\F(H^s)$ is characterized by
\begin{equation*}
  \F(H^s)=\{f\in \Sch'(R),\ \|f\|_{\F(H^s)}^2:=
  \int_{\R}\<x\>^{2s}|f(x)|^2dx<\infty\}, \quad \<x\>=\sqrt{1+x^2}.
\end{equation*}
The result of \cite{Ozawa91} also addresses the case where $u_-\in \H$
(defined below, see \eqref{eq:H}) where, provided again that
$\|\widehat u_-\|_{L^\infty}$ 
sufficiently small, the solution $u$ to \eqref{eq:NLS} has an
$H^1$ regularity, and  convergence holds in this space.

In \cite{HN98}, the asymptotic completeness was proven for $u_0\in
H^\gamma\cap \F(H^\gamma)$ with $\gamma>1/2$ and
$\|u_0\|_{H^\gamma\cap \F(H^\gamma)}$ sufficiently small. The obtained
asymptotic state $u_+$ is such that $\widehat u_+\in L^2\cap L^\infty$. 

Denote
\begin{equation}\label{eq:H}
  \begin{aligned}
{\H}&:=\{ f\in H^1(\R);\ \<x\>\d_x f, \<x\>^3f \in L^2(\R)\}\\
&= \{f\in {\mathcal S}'(\R); \|f\|_{\H}:=
\|\<x\>\d_xf  \|_{L^2}+
\|\<x\>^3f  \|_{L^2}
<\infty \}.
\end{aligned}
\end{equation}
In \cite{CaCMP}, the main result of \cite{Ozawa91} was adapted for
$u_-\in \H$ with $\|\widehat u_-\|_{L^\infty}$
sufficiently small, and the regularity of  the solution $u$ to
\eqref{eq:NLS} was proven to be at least $\Sigma$,
hence $u_0\in \Sigma$, making it possible to connect this 
result with the asymptotic completeness from \cite{HN98}, thus
defining a map $u_-\mapsto u_+$ from (a subset of) $\H$ to
$L^2\cap L^\infty$. When invoking the modified scattering operator, we
refer to this notion.

This gap in regularity between $u_-$ and $u_+$ was considerably
diminished in \cite{HN06}, where, with the same notations as above,
the authors consider the setting (along with smallness
conditions)
\begin{equation*}
  u_-\in \F(H^\alpha),\quad u_0\in \F(H^\beta),\quad u_+\in \F(H^\delta),
\end{equation*}
with the constraints $1/2<\delta<\beta<\alpha<1$, allowing these three
indices to be arbitrarily close one from another. This is
achieved by adapting the notion of (modified) asymptotic completeness
in order to avoid loss of differentiability issues, in space
dimensions two and three, where the borderline nonlinearity in terms
of scattering is $|u|^{2/d}u$.
\smallbreak

We also emphasize that there are many references addressing the theory
of long range scattering for nonlinear Schr\"odinger equations
(e.g. \cite{MR3724144,MR2047418,MR2864547,LindbladMurphy2006,MaMi18,MaMi19,MTT03}), as
well as for 
Schr\"odinger-like equations, like Hartree 
equation (e.g. \cite{MR1855975,MR3192651,MR3359525,Wada2000}), derivative nonlinear Schr\"odinger equation
(e.g. \cite{MR1618664,MR3144794}), Maxwell-Schr\"odinger system (e.g. \cite{MR2474176,MR2342882}),
wave-Schr\"odinger system (e.g. \cite{MR2853555}), not to mention other
dispersive equations.
\smallbreak

We note that the very definition of the modified wave and scattering
operators encodes the fact that the nonlinearity is cubic, and
recovering the nonlinearity from the scattering map does not make
sense, contrary to the case of \cite{MR4576319,KiMuVi-p}. In the case
where $\lambda$ is allowed to depend on $x$ in a somehow perturbative
way, Chen and Murphy \cite{ChenMurphy} showed that the inverse of the
modified wave operator uniquely determines $\lambda$. One of the tools
of the proof there is to study the behavior of this operator near the
origin, thanks to a rather implicit expansion
(\cite[Proposition~4.1]{ChenMurphy}). We present a 
more explicit formula in the next subsection, when
$\lambda$ is constant.

To fix the ideas, we summarize the above discussion by introducing the
following definition, where regularity aspects are left out for
simplicity. 
\begin{definition}
  Given $u_-$, the modified wave operator acting on $u_-$ is given by
  $u_{\mid t=0} =\W (u_-)$, where $u$ solves \eqref{eq:NLS}, and
  satisfies
  \begin{equation*}
  u(t,x)= e^{iS_-(t,x)}U(t)u_-(x) +o(1) \text{ in }L^2(\R)\text{ as }t\to-\infty,
\end{equation*}
where $S_-$ is defined in \eqref{eq:defS}.

The modified scattering
operator is defined by $\Scatt (u_-)=u_+$ if  the above solution $u$
satisfies in addition 
 \begin{equation*}
  u(t,x)= e^{iS_+(t,x)}U(t)u_+(x) +o(1) \text{ in }L^2(\R)\text{ as }t\to+\infty,
\end{equation*}
where $S_+$ is defined in \eqref{eq:defS}.
\end{definition}

\subsection{Main result}
In the same spirit as what has been achieved for the wave and
scattering operators in the short range case, we consider the
asymptotic behavior of the modified wave and scattering operators,
with two restrictions compared to \cite{CaGa09}. First, we shall
confine ourselves to the asymptotic expansion near the origin. Second, we
compute only the first two terms of this asymptotic expansion. We will
see that this already requires some amount of work, but that the
same method should provide some, if not all,
higher order terms in the expansion at the origin. On the other hand,
the description of these operators near a nontrivial state certainly
requires a different approach. \smallbreak

We emphasize that the question addressed here is different from  the
(higher order) asymptotic expansion of the large time 
behavior of $u(t)$, as studied in \cite{MR1913680}.
Our main results are gathered in the
following statement:

\begin{theorem}\label{theo:main}
  Let $v_-\in \H$ and $v_0\in \Sigma$.
  \begin{itemize}
  \item For any $0<\eta<2$, we have, in $\Sigma$ and  as $\eps\to 0$,
    \begin{equation*}
      \W(\eps v_-) = \eps v_- + \eps^3 w_2 +\O\(\eps^{5-\eta}\),
    \end{equation*}
    where $w_2\in \Sigma$ is defined by
     \begin{align*}
   w_2 & = -i\lambda\int_{-\infty}^{-1}\(  U(-\tau)\(
      |U(\tau)v_-|^2U(\tau)v_-\) +\frac{1}{|\tau|} \F^{-1}\( 
      |\widehat v_-|^2\widehat v_-\)\) d\tau\\
&\quad  -i\lambda\int_{-1}^0 U(-\tau) \( 
 |U(\tau)v_-|^2U(\tau)v_-\)d\tau.
     \end{align*}
      \item For any $0<\eta<2$, we have, in $L^2$ and  as
      $\eps\to 0$,
      \begin{equation*}
        \(W_+^{\rm mod}\)^{-1} (\eps v_0) = \eps v_ 0 + \eps^3 \mu_2+
        \O\(\eps^{5-\eta}\), 
      \end{equation*}
      where $\mu_2=\mu_2(v_0)\in L^2$ is defined by
      \begin{align*}
         \mu_2& =  -i\lambda \int_0^1 \( U(-s)\(
                |U(s)v_0|^2U(s)v_0\)\)ds \\
        &\quad +\lambda\int_1^\infty 
      \( M(-\tau)\F^{-1}\(\left|\widehat{M(\tau)v_0}\right|^2
       \widehat{M(\tau)v_0} \)- \F^{-1}\(|\widehat v_0 |^2\widehat
          v_0\)\) \frac{d\tau}{\tau} ,
     \end{align*}
   and $M(t)$ stands for the multiplication by
     $e^{i\frac{x^2}{2t}}$. 
 \item For any $0<\eta<2$, we have, in $L^2$ and  as $\eps\to 0$,
     \begin{equation*}
       \Scatt(\eps v_-) = \eps v_- + \eps^3  \nu_2 + \O\(\eps^{5-\eta}\),
     \end{equation*}
     where $\nu_2\in L^2$ is defined by
 $  \nu_2 = w_2 +  \mu_2(v_-),$
     where $w_2$ is given by the first point, and $\mu_2$ by the
     second. 
   \end{itemize}
\end{theorem}
\begin{remark}
  The functions $\mu_2$ and $\nu_2$ are more regular than merely
  $L^2$, as we will see in Section~\ref{sec:complete-higher} that they
  belong to $H^\gamma\cap \F(H^\gamma)$ for any $0<\gamma<1$.
\end{remark}
It may be surprising that the leading order behavior of the modified
wave and scattering operators at the origin is the identity: we recall
that from their definition, these operators are already nontrivial,
and account for long range effects. We also emphasize that the first
corrector has a more involved expression than in the short range case,
a case where we would simply have (typically for the two-dimensional
and three-dimensional
cubic Schr\"odinger equations, see e.g. \cite{CaGa09})
\begin{align*}
  w_2& =-i\lambda \int_{-\infty}^0 \( U(-s)\(
       |U(s)v_0|^2U(s)v_0\)\)ds ,\\
  \mu_2& = -i\lambda \int_0^\infty \( U(-s)\(
        |U(s)v_0|^2U(s)v_0\)\)ds . 
\end{align*}
Note that in the (one-dimensional) long range case, these integrals diverge. 
\smallbreak

As evoked above, in principle, the method of proof that we present
allows to compute (some) higher order terms in the asymptotic
expansions near the origin. 

\subsection{Outline}

In Section~\ref{sec:prelim}, we recall several properties which are
classical in the context of scattering theory for nonlinear
Schr\"odinger equations, and which are of constant use in this
paper. In Section~\ref{sec:wave-op}, we revisit the construction of
the modified wave operator $\W$, in such a way that the leading order
behavior of 
this operator at the origin, presented in Section~\ref{sec:wave-zero}, is rather
straightforward.  In Section~\ref{sec:wave-zero}, we also consider the
first corrector in the asymptotic expansion of $\W$ at the origin, an
aspect which requires some extra work. In Section~\ref{sec:complete},
we recall the modified asymptotic completeness result established in
\cite{HN98}, and infer the leading of behavior of $\Scatt$ at the
origin. The first corrector is derived in
Section~\ref{sec:complete-higher}, where the last two error estimates
announced in Theorem~\ref{theo:main} are proved. 
\subsection{Notations}
We recall the classical factorization of the Schr\"odinger group, 
\begin{equation*}
  U(t)= e^{i\frac{t}{2}\d_x^2}= M(t)D(t)\F M(t),
\end{equation*}
where the multiplication $M(t)$, the dilation $D(t)$ and
the Fourier transform $\F$ are defined by
\begin{equation}\label{eq:MDF}
  \begin{aligned}
    & M(t) = e^{i\frac{ x^2}{2t}}, \quad D(t) f(x) =
    \frac{1}{(it)^{1/2}}f\(\frac{x}{t}\),\\
    &{\F} f(\xi)=\widehat
  f(\xi)=\frac{1}{\sqrt{2\pi}}\int_{\R}f(x)e^{-ix \xi} dx.
  \end{aligned}
 \end{equation}
Note that  each of these three operators is unitary on $L^2(\R)$, and
that \eqref{eq:asym-lin} reads $U(t)\approx M(t)D(t)\F $, see
also Lemma~\ref{lem:R}.
We recall that the space $\Sigma$ is defined by
\begin{equation*}
  \Sigma=H^1\cap \F(H^1)=\left\{f\in H^1(\R);\
    \|f\|_\Sigma:= \|f\|_{L^2}+\|\d_x f\|_{L^2} + \|x
    f\|_{L^2}<\infty\right\}.
\end{equation*}
We note that $\Sigma$ is a Banach algebra, invariant under the Fourier
transform, and such that $\Sigma\hookrightarrow L^1\cap L^\infty$. In
addition, $\Sigma$ is invariant under the action of $U(t)$, for any
$t\in \R$.
\smallbreak

For functions $f^\eps,g^\eps\ge 0$ depending on time $t$ and $\eps$, the
notation
\begin{equation*}
  f^\eps\lesssim g^\eps
\end{equation*}
means that there exists $C$ independent of $t$ and $\eps$ such that
\begin{equation*}
  f^\eps\le C g^\eps. 
\end{equation*}

\section{Technical preliminaries}
\label{sec:prelim}

In this section, we gather classical estimates which can be found in
several references cited in the introduction.

\begin{lemma}\label{lem:J}
The operator
\[J(t)=x+it\d_x\]
satisfies the following properties:
  \begin{itemize}
  \item $J(t) =U(t)xU(-t)$, and therefore $J$ commutes with the
 linear part of \eqref{eq:NLS},
\begin{equation}\label{eq:Jcommute}
\left[ J(t),i\d_t +\frac{1}{2}\d_x^2\right]=0\, .
\end{equation}
\item It can be factorized as
 \begin{equation*}
J(t)= i t \, e^{i\frac{x^2}{2t}}\d_x\Big( e^{-i\frac{x^2}{2t}}\,
\cdot\Big)\, .
\end{equation*}
As a consequence, $J$ yields weighted Gagliardo-Nirenberg
inequalities. For
$2\le r \le \infty$, there exists
$C(r)$ depending only on $r$ such that
\begin{equation}\label{eq:GNlibre}
\left\| f \right\|_{L^r}\le \frac{C(r)}{|t|^{\delta(r)}} \left\|
f \right\|_{L^2}^{1-\delta(r)} \left\|
J(t) f \right\|_{L^2}^{\delta(r)},\quad \delta(r):=
\frac{1}{2}-\frac{1}{r}\, .
\end{equation}
Also, if $F(z)=G(|z|^2)z$ is $C^1$, then $J(t)$
acts like a derivative on $F(w)$:
\begin{equation}\label{eq:Jder}
J(t)\(F(w)\) = \d_z F(w)J(t)w -\d_{\overline z} F(w)\overline{ J(t)w
}\, .
\end{equation}
\end{itemize}
\end{lemma}

\begin{lemma}\label{lem:R}
  Denote by
  \begin{equation*}
    \cR(t)= M(t)D(t)\F \(M(t)-1\)\F^{-1},
  \end{equation*}
  where the above terms are defined in \eqref{eq:MDF}. The following
  estimates hold for $|t|\ge 1$:
\begin{enumerate}
\item For all $s>1/2$ and $0\le \theta \le 1$,
  \begin{equation*}
    \|\cR(t)f\|_{L^\infty_x}\lesssim
    \frac{1}{|t|^{1/2+\theta}}\|f\|_{H^{s+2\theta}},\quad \forall
      f\in \Sch(\R).
    \end{equation*}
\item For all $0\le \theta \le 1$,
  \begin{equation*}
    \|\cR(t)f\|_{L^2_x}\lesssim
    \frac{1}{|t|^{\theta}}\|f\|_{H^{2\theta}},\quad \forall
      f\in \Sch(\R).
    \end{equation*}    
\item For all $0\le \theta \le 1$,
  \begin{equation*}
    \|J(t)\cR(t)f\|_{L^2_x}\lesssim
    \frac{1}{|t|^{\theta}}\|f\|_{H^{1+2\theta}},\quad \forall
      f\in \Sch(\R).
    \end{equation*}
\item We have
  \begin{equation*}
    \|\d_x \cR(t)f\|_{L^2_x}\lesssim
    \frac{1}{|t|}\|f\|_{H^1}+
    \frac{1}{|t|}\|x f\|_{H^{2}},\quad \forall
      f\in \Sch(\R).
    \end{equation*}    
\end{enumerate}
\end{lemma}
\begin{remark}
  We emphasize that $\cR(t)$ does not map $L^\infty(\R)$ into itself,
  since $\F \(M(t)-1\)\F^{-1}= U(\frac{1}{t})-1$ and 
  $U(\frac{1}{t})(e^{-it\frac{x^2}{2}})=\delta_0$. 
\end{remark}
Even though such estimates can be found in the existing literature,
we give a proof, as the ideas will be resumed when $M(t)-1$ is present. 
\begin{proof}
  In view of the definition of $\cR$, we readily have
  \begin{equation*}
    \|\cR(t)f\|_{L^\infty_x} =\frac{1}{|t|^{1/2}}\|\F (M(t)-1)\F^{-1}f\|_{L^\infty}
  \end{equation*}
  Using Hausdorff-Young inequality and the general estimate
  \begin{equation}\label{eq:sinus}
    |M(t)-1| = 2\left|\sin\(\frac{x^2}{4t}\)\right|\lesssim
    \left|\frac{x^2}{t}\right|^\theta,\quad \forall \, 0\le \theta\le 1,
  \end{equation}
  we find
  \begin{equation*}
    \|\cR(t)f\|_{L^\infty_x} \lesssim \frac{1}{|t|^{1/2}}\left\|
    \left|\frac{x^2}{t}\right|^\theta  \F^{-1}f\right\|_{L^1}. 
    \end{equation*}
 Using the easy property $\F(H^s)\hookrightarrow L^1$ for
 $s>1/2$ (the same is true on $\R^d$ provided that $s>d/2$,
 from Cauchy-Schwarz inequality),
 \begin{equation*}
    \|\cR(t)f\|_{L^\infty_x} \lesssim \frac{1}{|t|^{1/2+\theta}}\left\|
    \<x\>^s|x|^{2\theta} \F^{-1}f\right\|_{L^2}\lesssim
    \frac{1}{|t|^{1/2+\theta}} \|f\|_{H^{s+2\theta}}, 
  \end{equation*}
  hence the first inequality. For the second one, since $D(t)$ is
  unitary on $L^2$,  Plancherel formula
  and \eqref{eq:sinus} yield
  \begin{equation*}
     \|\cR(t)f\|_{L^2_x} = \| (M(t)-1)\F^{-1}f\|_{L^2}\lesssim \frac{1}{|t|^{\theta}} \left\|
   |x|^{2\theta} \F^{-1}f\right\|_{L^2} \lesssim
 \frac{1}{|t|^{\theta}} \|f\|_{H^{2\theta}}.
  \end{equation*}
For the third inequality, we use the formula $J(t)=U(t)xU(-t)$, along
with the factorization
\begin{equation}\label{eq:U(-t)}
  U(-t) = iM(-t) \F^{-1}D\(\frac{1}{t}\) M(-t),
\end{equation}
to obtain
 \begin{equation*}
     \|J(t)\cR(t)f\|_{L^2_x} = \|x (M(t)-1)\F^{-1}f\|_{L^2}\lesssim
     \frac{1}{|t|^{\theta}} \left\| 
   |x|^{1+2\theta} \F^{-1}f\right\|_{L^2} \lesssim 
 \frac{1}{|t|^{\theta}} \|f\|_{H^{1+2\theta}}.
  \end{equation*}
  To estimate $\d_x \cR(t)f$, two terms appear, whether the derivative
  hits the first factor $M(t)$ or not, and we have
  \begin{equation*}
     \|\d_x\cR(t)f\|_{L^2_x} \le  \left\|\frac{x}{t}D(t)\F
       (M(t)-1)\F^{-1}f\right\|_{L^2}+ \frac{1}{|t|}\left\|\d_x\F
       (M(t)-1)\F^{-1}f\right\|_{L^2}.
   \end{equation*}
   The first term on the right hand side is equal to
   \begin{align*}
     \left\|x\F
       (M(t)-1)\F^{-1}f\right\|_{L^2}&=  \left\|
  \d_x (M(t)-1)\F^{-1}f\right\|_{L^2}\\
    &\le \left\| \frac{x}{t}\F^{-1}f\right\|_{L^2}
   +\left\|
  (M(t)-1) \d_x\F^{-1}f\right\|_{L^2}\\
   & \lesssim \frac{1}{|t|}
    \|f\|_{H^1} + \frac{1}{|t|^{\theta}}\|xf\|_{H^{2\theta}},
  \end{align*}
  for any $\theta\in [0,1]$. Note that 
 \[ \frac{1}{|t|}\left\|\d_x\F
     (M(t)-1)\F^{-1}f\right\|_{L^2}= \frac{1}{|t|}\left\|x
     (M(t)-1)\F^{-1}f\right\|_{L^2}\lesssim \frac{1}{|t|}\left\|x
   \F^{-1}f\right\|_{L^2},
\]
hence the result, by choosing $\theta=1$.
\end{proof}

\section{Modified wave operator}
\label{sec:wave-op}
A key aspect in the proofs of the existence of modified wave
operators for \eqref{eq:NLS}, in particular to get minimal regularity
assumptions on the asymptotic state $u_-$, is to consider a suitable
approximate solution near $t=-\infty$. Introduce two such approximate
solutions, $u_1$ and $u_2$, defined for $t\le -10$, given by
\begin{equation*}
  u_1(t,x) :=\frac{1}{(it)^{1/2}}e^{i\frac{x^2}{2t}}\fum
  \(\frac{x}{t}\) e^{i S_-(t,x)}=M(t)D(t)\widehat w(t,x),
\end{equation*}
where
\begin{equation*}
\widehat w(t):= \fum e^{i\lambda |\fum|^2\log|t|},
\end{equation*}
where we resume the same notations as in \cite{HN06},
and
\begin{equation*}
  u_2(t) :=U(t)\F^{-1} \widehat w.
\end{equation*}
As $t\to -\infty$, $u_1$ and $u_2$ are close, in view of
Lemma~\ref{lem:R}, since $u_2(t)-u_1(t)= \cR(t)\widehat w$.
Our goal is to keep track of the smallness of the asymptotic state as precisely as
possible, in the sense that we consider $u_-^\eps=\eps v_-$ for some
fixed $v_-$,  assume that $\eps>0$ is small, and we want to get larger
powers of $\eps$ in the error terms whenever we can. It turns out that
apparently, $u_2$ is a better candidate than $u_1$ in this
direction. More explicitly, the Duhamel formula for $u-u_1$ (see
\cite{HN06}) contains a linear term ($R\widehat w$), which we want to
remove. We first gather some estimates on $u_1$.
\begin{lemma}\label{lem:u1}
  Let $\um\in \H$: $\|u_1(t)\|_{L^2}=\|u_2(t)\|_{L^2}=\|\um\|_{L^2}$ for all $t\le
  -10$. 
  There exists $C>0$ independent of $\um$ such that for all $t\le -10$,
  \begin{align*}
    & \| u_1(t)\|_{L^\infty}\le
      \frac{\|\fum\|_{L^\infty}}{\sqrt{|t|}},\\
    &\|\d_x u_1(t)\|_{L^2}\le C\(\|\um\|_{\H}
    +\|\um\|_{\H}^3\) ,\\
    &\|J(t) u_1(t)\|_{L^2}\le C  \|\um\|_{\H}
  \(1+ \|\um\|_{\H}^2\log |t|\).  
  \end{align*}
\end{lemma}
\begin{proof}
  The conservation of the $L^2$-norm is obvious, as well as 
the estimate 
\begin{equation*}
  \|u_1(t)\|_{L^\infty}\le \frac{\|\fum\|_{L^\infty}}{\sqrt{|t|}}.
\end{equation*}
The expression of $u_1$ yields directly
\begin{equation*}
  \|\d_x  u_1(t)\|_{L^2} \le \|\d_x \fum\|_{L^2} +
  |\lambda|\frac{\log |t|}{|t|}\|\fum\|_{L^\infty}^2\|\d_x \fum\|_{L^2} .
\end{equation*}
In view of the factorization for $J$ stated in Lemma~\ref{lem:J},
\begin{align*}
  J(t)u_1(t)&= it M(t) \d_x \( D(t)\fum
                    e^{iS_-(t)}\)\\
&= e^{i\frac{x^2}{2t}+iS_-(t)}\(i D(t)(\d_x \fum)-\lambda 
   D(t)\fum \times \log|t| \times\d_x \left|\fum\right|^2\(\frac{x}{t}\)\),
\end{align*}
thus
\begin{equation*}
  \| J(t)u_1(t)\|_{L^2}\lesssim \|\fum\|_{H^1} + 
\log |t|\|\fum\|_{L^\infty}^2 \|\fum\|_{H^1} .
\end{equation*}
The lemma follows. 
\end{proof}

\begin{lemma}\label{lem:Rw}
  Let $\um\in \H$. 
  There exists $C>0$ independent of $\um$ such that for all $t\le -10$,
  \begin{align*}
    & \| \cR(t)\widehat w\|_{L^\infty}\le
      \frac{C}{\sqrt{|t|}} \(\|\fum\|_{H^1}+ \|\fum\|_{H^1}^3\),\\
    &  \|\cR(t)\widehat w\|_{L^\infty}\le \frac{C}{|t|^{3/2}}\(\|
   \um\|_{\H}+ (\log|t|)^{3}\|\um\|_{\H}^7\).\\
    &\|  U(-t) \cR(t)\widehat w\|_{\Sigma}\le \frac{C}{|t|}\( \|\um\|_{\H} +
  (\log|t|)^{3}\| \um\|_{\H}^7\).
  \end{align*}
\end{lemma}
\begin{proof}
   If we use
the first point of Lemma~\ref{lem:R} with $\theta=0$ and
$s>1/2$, the phase factor causes the appearance of a $\log|t|$ term,
since the $H^1$-norm of $\widehat w$ is involved. So we rather take
$\theta>0$ and $s>1/2$ such that $s+2\theta=1$, and so,
for $t\le -10$,
\begin{align*}
  \|\cR(t)\widehat w\|_{L^\infty}\lesssim  \frac{\|\widehat
    w\|_{H^1}}{|t|^{1/2+\theta}}& \lesssim \frac{ 1}{|t|^{1/2+\theta}}
\(\|\fum\|_{H^1}+(\log|t|)\|\fum\|_{L^\infty}^2\|\fum\|_{H^1}\)\\ 
&\lesssim \frac{1}{\sqrt{|t|}}\(\|
   \fum\|_{H^1}+\|\fum\|_{H^1}^3\),
\end{align*}
where the logarithmic factor was left out since $\theta>0$.
 It is actually at the level of this estimate that
the $H^1$-norm of $\fum$ appears, instead of merely its
$L^\infty$-norm.

To obtain the stronger time decay, we choose $s=\theta=1$ in the first
point of Lemma~\ref{lem:R}, hence, since $H^3(\R)$ is
an algebra, 
\begin{align*}
  \|\cR(t)\widehat w\|_{L^\infty}\lesssim  \frac{\|\widehat
    w\|_{H^3}}{|t|^{3/2}}& \lesssim \frac{ 1}{|t|^{3/2}}
\(\|\fum\|_{H^3}+(\log|t|)^3\|\fum\|_{H^3}^7\).
\end{align*}

In view of Lemma~\ref{lem:J},
\begin{equation*}
  \|U(-t) \cR\widehat w\|_{\Sigma}= \|\cR\widehat w\|_{L^2}+\|\d_x
 \cR\widehat w\|_{L^2} + \|J(t)\cR\widehat w\|_{L^2},
\end{equation*}
and Lemma~\ref{lem:R} yields
\begin{equation*}
  \|U(-t) \cR(t)\widehat w\|_{\Sigma}\lesssim \frac{1}{|t|}\|\widehat
  w\|_{H^3} + \frac{1}{|t|}\|x\widehat
  w\|_{H^2} .
\end{equation*}
Using the fact that $H^s(\R)$ is an algebra for $s>1/2$, we infer
\begin{align*}
   \|U(-t) \cR(t)\widehat w\|_{\Sigma}&\lesssim
   \frac{1}{|t|}\(\|\fum\|_{H^3} + (\log|t|)^3
   \|\fum\|_{H^3}^7 \)\\
&\quad +  \frac{1}{|t|}\( \|x\fum\|_{H^2} +
  (\log|t|)^2\|x\fum\|_{H^2}^5\)\\
&\lesssim \frac{1}{|t|}\( \|\um\|_{\H} +
  \sum_{j=0,1}(\log|t|)^{2+j}\| \um\|_{\H}^{5+2j}\),
\end{align*}
hence the lemma. 
\end{proof}

\begin{proposition}\label{prop:wave-op}
There exist $\delta_0>0$, $C$, and a polynomial $P$,  such that the following holds.
  For any $u_-\in L^2$ with $\um\in \H$, where $\H$ is
  defined in \eqref{eq:H}, and such that $\|
  \fum\|_{H^1}\le \delta_0$, there exists  a unique
  $u_0\in \Sigma$ such that the solution $u\in C(\R,\Sigma)$
  to \eqref{eq:NLS} with 
  $u_{\mid t=0}=u_0$ satisfies, for $t\le -10$,
  \begin{equation*}
    \left\|U(-t)\(u(t)-u_2(t)\)\right\|_{\Sigma}\le
    C\|\um\|_{\H}^3P\(\|\um\|_{\H}\) \frac{(\log|t|)^4}{|t|}. 
  \end{equation*}
\end{proposition}
\begin{remark}
  The choice of the approximate solution $u_2$ is tailored to ensure
that, with the proof presented below, the error term is superlinear in
some norm of $u_-$, uniformly in the limit $t\to
-\infty$. 
On the other hand, the smallness assumption, which is a consequence of
the first estimate in 
Lemma~\ref{lem:Rw},  is stronger than in the
cited references, since Sobolev embedding yields $H^1(\R)\subset
L^\infty(\R)$. 
\end{remark}
\begin{proof}
  Like in \cite{Ozawa91} or \cite{HN06}, the proof relies on a fixed
  point argument near $t=-\infty$. Resuming the computations from
  \cite{HN06} (and keeping the same notations), 
  \begin{equation*}
    i\d_t \widehat w =\frac{\lambda}{t}|\widehat w|^2\widehat w,
  \end{equation*}
and so 
we check that $u_2$ solves the equation
  \begin{equation*}
    i\d_t \F U(-t) u_2 = \lambda \F
    U(-t)\(\frac{1}{t}U(t)\F^{-1}\(|\widehat w|^2\widehat w\)\).
  \end{equation*}
Therefore, we want to solve
\begin{equation*}
    i\d_t \F U(-t)\( u-u_2\) = \lambda \F
    U(-t)\(|u|^2 u-\frac{1}{t}U(t)\F^{-1}\(|\widehat w|^2\widehat w\)\).
\end{equation*}
We insert the term $|u_2|^2 u_2$ into the above equation,
and  use the identity
\begin{align*}
  \frac{1}{t}U(t)\F^{-1}\(|\widehat w|^2\widehat w\)&= \frac{1}{t}MD\F M\F^{-1}
  \(|\widehat w|^2\widehat w\) \\
&= \frac{1}{t}\cR(t)\( |\widehat w|^2\widehat w\)+
  \frac{1}{t}M(t)D(t)\(|\widehat w|^2\widehat w\), 
\end{align*}
with 
\begin{align*}
  \frac{1}{t}M(t)D(t)\(|\widehat w|^2\widehat w\)= M(t)|M(t)D(t)\widehat
  w|^2D(t)\widehat w = |u_1|^2u_1,
\end{align*}
so Duhamel's formula reads, along with the requirement
$U(-t)\(u(t)-u_2(t)\)\to 0 $ as $t\to -\infty$,
\begin{equation}\label{eq:Duhamel1}
  \begin{aligned}
    u(t)=u_2(t)& - i\lambda\int_{-\infty}^t
    U(t-\tau)\(|u|^2u- |u_2|^2 u_2\)(\tau)d\tau\\
&\quad +i\lambda \int_{-\infty}^t
    U(t-\tau)\( |u_2|^2 u_2-|u_1|^2u_1\)(\tau)d\tau\\
 &\quad  -i\lambda \int_{-\infty}^t
    U(t-\tau)\(\frac{1}{\tau}\cR(\tau)\(|\widehat w|^2\widehat w\) \)(\tau)d\tau.
  \end{aligned}
\end{equation}
The last two lines correspond to source terms, involving only the
various approximate solutions, and can be estimated thanks to
Lemma~\ref{lem:R}, as $u_2 = u_1+\cR\widehat w$. 

Denote by $\Phi(u)= u_2+\Phi_1(u)+\Phi_2+\Phi_3$ the right hand side of
\eqref{eq:Duhamel1}, where $\Phi_j$ corresponds to the $j$-th
line (note that $\Phi_2$ and $\Phi_3$ do not depend on $u$). We have
\begin{align*}
  \left|\d_x\(|u_2|^2 u_2-|u_1|^2u_1\)\right|&=  \left|\d_x\(|\cR\widehat w
  +u_1|^2 (\cR\widehat w+u_1)-|u_1|^2u_1\)\right| \\
&\lesssim \(|\cR\widehat w|^2+|u_1|^2\) |\d_x \cR\widehat w| + \(|\cR\widehat w|+|u_1|\)
     |\cR\widehat w| |\d_x  u_1|. 
\end{align*}
Using the formula $J(t)=U(t)xU(-t)$ from Lemma~\ref{lem:J}, as well as
\eqref{eq:Jder} that state that $J(t)$ can be thought of as a
derivative here, we find
\begin{align*}
  \left\| U(-t)\Phi_2(t)\right\|_{\Sigma}&\lesssim \int_{-\infty}^t\(\|\cR(\tau)\widehat
  w\|_{L^\infty}^2+\|u_1(\tau)\|_{L^\infty}^2\) \|U(-\tau)\cR\widehat
 w\|_{\Sigma}d\tau\\
&
 + \int_{-\infty}^t\(\|\cR(\tau)\widehat w\|_{L^\infty}+\|u_1(\tau)\|_{L^\infty}\) 
     \|\cR(\tau)\widehat w\|_{L^\infty} \|U(-\tau)  u_1(\tau)\|_{\Sigma}d\tau.
\end{align*}
We invoke Lemmas~\ref{lem:u1} and \ref{lem:Rw} to estimate the above
terms. More precisely, then factor $\|\cR(\tau)\widehat w\|_{L^\infty}
$ is controlled by the second estimate of Lemma~\ref{lem:Rw}, and we
get, for $\|\fum\|_{H^1}\le 1$ and $t\le -10$:
\begin{equation*}
  \|U(-t)\Phi_2(t)\|_\Sigma\lesssim
  \|\fum\|_{H^1}\|\um\|_{\H}^2P_2\( 
  \|\um\|_{\H}\) \int_{-\infty}^t \frac{(\log |\tau|)^4}{\tau^2}d\tau,
\end{equation*}
for some polynomial $P_2$. 
The term $\Phi_3$ is estimated thanks to Lemma~\ref{lem:R}, and the
choice of parameter is motivated by the previous estimate, keeping in
mind that a factor $\frac{1}{\tau}$ is already present in the
definition of $\Phi_3$. For the $L^2$-estimate, we thus choose
$\theta=1$ in Lemma~\ref{lem:R}, so that
\begin{align*}
  \|\Phi_3(t)\|_{L^2}&\lesssim \int_{-\infty}^t \left\| |\widehat w|^2\widehat
    w\right\|_{H^2}\frac{d\tau}{\tau^2 }\lesssim \int_{-\infty}^t \|
  \widehat w \|_{L^\infty}^2 \| \widehat
                       w\|_{H^2}\frac{d\tau}{\tau^2 }\\
  &\lesssim \|\fum\|_{L^\infty}^2\|\fum \|_{H^2}
    \int_{-\infty}^t
   \(1+ (\log|\tau|)^2 \|\fum \|_{H^2}^4\) \frac{d\tau}{\tau^2 }.
\end{align*}
The norm  $\|xU(-t)\Phi_3(t)\|_{L^2}=\|J(t)\Phi_3(t)\|_{L^2}$ is
estimate similarly, by setting again $\theta=1$ (in the third point of
Lemma~\ref{lem:R}, and we get, since the $H^2$-norm is replaced by an
$H^3$-norm in the above computation,
\begin{equation*}
  \|J(t)\Phi_3(t)\|_{L^2}\lesssim \|\fum\|_{L^\infty}^2\|\fum \|_{H^3}
    \int_{-\infty}^t
   \(1+ (\log|\tau|)^3\|\fum \|_{H^3}^6\) \frac{d\tau}{\tau^2 }.
 \end{equation*}
 The choice of suitable parameters for the $\dot H^1$-norm was already
 made in the statement of Lemma~\ref{lem:R}, and we have
 \begin{equation*}
  \|\d_x\Phi_3(t)\|_{L^2}\lesssim 
    \int_{-\infty}^t
   \(\||\widehat w|^2 \widehat w\|_{H^1} + \|x|\widehat w|^2 \widehat w\|_{H^2}\)
   \frac{d\tau}{\tau^2 }. 
 \end{equation*}
 The term involving the $H^1$-norm has been estimated above, so we
 focus on the other norm. Direct computations yield
 \begin{equation*}
  \left| \d_x^2\( x|\widehat w|^2 \widehat w\)\right| \lesssim  
     |x||\widehat w|^2 |\d_x^2\widehat w| + |\widehat w|^2 |\d_x\widehat w| +
     |x||\widehat w||\d_x\widehat w|^2,
   \end{equation*}
   and we get
   \begin{align*}
  \|\d_x\Phi_3(t)\|_{L^2}&\lesssim 
  \|\fum\|_{L^\infty}^2 \|\fum\|_{H^1}  \int_{-\infty}^t
   \(1+ (\log|\tau|)\|\fum\|_{H^1}^2\)
    \frac{d\tau}{\tau^2 }\\
     &\quad +\|\fum\|_{L^\infty}\|\um\|_{\H}^2\int_{-\infty}^t
     \(1+ (\log|\tau|)^2 \|\um\|_{\H}^4\)     \frac{d\tau}{\tau^2 }.                   
   \end{align*}
   The source term is therefore controlled, for $t\le -10$, by:
   \begin{equation}\label{eq:sourcePhi}
     \|U(-t)(\Phi_2(t)+\Phi_3(t))\|_{\Sigma}\lesssim
     \|\fum\|_{H^1}\|\um\|_{\H}^2P\(\|\um\|_{\H}\)\int_{-\infty}^t
      \frac{\(\log|\tau|\)^4}{\tau^2}d\tau,
   \end{equation}
   for some polynomial $P$ whose precise expression is irrelevant.

   The end of the proof consists of a fixed point argument. Mimicking
   \cite{Ozawa91}, for $\alpha\in ]1/2,1[$ arbitrary (but fixed),
  and $T\gg 1$, introduce the
   space
   \begin{align*}
     X^\alpha(T)&=\Bigl\{ u\in L^\infty(]-\infty,-T],L^2) \text{ such that }t\mapsto U(-t)u(t)\in
          L^\infty(]-\infty,-T],\Sigma), \\
   &\qquad   \sup_{t\le
     -T}|t|^\alpha\|U(-t)(u(t)-u_2(t))\|_{\Sigma}\le    \|\fum\|_{H^1}\Bigr\}.
   \end{align*}
   and define on $ X^\alpha(T)$ the metric
   \begin{equation*}
    d(u,v) = \sup_{t\le
     -T}|t|^\alpha\|u(t)-v(t)\|_{L^2}.
 \end{equation*}
 Note that here, we choose to measure distance by considering the
 $L^2$ norm only: $X^\alpha(T)$, equipped with this distance, is
 a complete metric space. 
 
  For $u\in  X^\alpha(T)$, $\Phi_1(u)$ is estimated by
  \begin{align*}
    \|U(-t)\Phi_1(u(t))\|_{L^2}&\lesssim \int_{-\infty}^t
    \(\|u(\tau)\|_{L^\infty}^2+
    \|u_2(\tau)\|_{L^\infty}^2\)\|u(\tau)-u_2(\tau)\|_{L^2}d\tau \\
    &\lesssim
      \|\fum\|_{H^1}^2\int_{-\infty}^t\frac{d\tau}{\tau^{1+\alpha}}\lesssim \frac{ \|\fum\|_{H^1}^2}{|t|^\alpha}.
  \end{align*}
  For $A\in\{\d_x,J(t)\}$, we have
  \begin{equation}\label{eq:est-gen}
   \left|A\( |u|^2u-|u_2|^2u_2\)\right| \lesssim |u|^2 |A(u-u_2)|+
                                          \(|u|+|u_2|\)|A u_2||u-u_2|.  
  \end{equation}
  Therefore,
  \begin{align*}
     \|U(-t)\Phi_1(u(t))\|_{\Sigma}&\lesssim \int_{-\infty}^t
    \|u(\tau)\|_{L^\infty}^2\|U(-\tau)(u(\tau)-u_2(\tau))\|_{\Sigma}d\tau
    \\
    + \int_{-\infty}^t
      (\|u(\tau)\|_{L^\infty} &+\|u_2(\tau)\|_{L^\infty})\|U(-\tau)u_2(\tau)\|_{\Sigma}
   \|u(\tau)-u_2(\tau)\|_{L^\infty}
      d\tau\\
  &\lesssim \|\fum\|_{H^1}^2P_3\(\|\um\|_{\H}\)\int_{-\infty}^t
    \frac{\log|\tau|}{\tau^{1+\alpha}}d\tau , 
  \end{align*}
 for some polynomial $P_3$ whose precise expression is irrelevant,
 where we have used Lemmas~\ref{lem:u1} and \ref{lem:Rw} to estimate 
  $\|U(-\tau)u_2(\tau)\|_{\Sigma}$. The 
  $L^\infty$-norm of $u(\tau)-u_2(\tau)$ is controlled by
  $|\tau|^{-1/2-\alpha}$ in view of Lemma~\ref{lem:J} and the
  definition of $X^\alpha(T)$. Therefore, by choosing $T$
  sufficiently large, we check that $\Phi$ maps $X^\alpha(T)$ to
  itself.

  To conclude by a fixed point argument, we show that
  $\Phi$ is a contraction provided
  that $\|\fum\|_{H^1}$ is sufficiently small. Indeed, for $u,v\in
  X^\alpha(T)$, 
  \begin{equation*}
    \Phi(u)(t)-\Phi(v)(t) =- i\lambda \int_{-\infty}^t
    U(t-\tau)\(|u|^2u-|v|^2v\)(\tau)d\tau, 
  \end{equation*}
  and so, for $t\le -10$,
  \begin{align*}
    \left\|\Phi(u)(t)-\Phi(v)(t)\right\|_{L^2}&\lesssim \int_{-\infty}^t
    \(\|u(\tau)\|_{L^\infty}^2+ \|v(\tau)\|_{L^\infty}^2\)\|u(\tau)-v(\tau)\|_{L^2}d\tau
    \\
   &\lesssim \|\fum\|_{H^1}^2 d(u,v) \int_{-\infty}^t\frac{d\tau}{\tau^{1+\alpha}},
  \end{align*}
  where we have used
  \begin{align*}
    \|u(\tau)\|_{L^\infty}+ \|v(\tau)\|_{L^\infty}& \le
    2\|u_2(\tau)\|_{L^\infty} + C\|u(\tau)-u_2(\tau)\|_{H^1}+
                                                    C\|v(\tau)-u_2(\tau)\|_{H^1}\\
    &\lesssim \frac{\|\fum\|_{H^1} }{\sqrt{|\tau|}},
  \end{align*}
  since $\alpha>1/2$. Therefore,
  \begin{equation*}
    d\(\Phi(u),\Phi(v)\)\lesssim \|\fum\|_{H^1}^2 d(u,v), 
  \end{equation*}
  hence the result provided that $\|\fum\|_{H^1}$ is
  sufficiently small.
 \end{proof}

\section{Behavior of the modified wave operator near the origin}
\label{sec:wave-zero}

\subsection{Leading order asymptotic behavior}

Let $u_-^\eps=\eps v_-$ with $ v_-\in \H$ independent of
$\eps$. Denote by $u^\eps$ the solution provided by
Proposition~\ref{prop:wave-op}. Then like $u_-^\eps$, $u^\eps$ is of order
$\eps$, and the remainder in Proposition~\ref{prop:wave-op} is
$\O(\eps^3)$. The long range phase correction $S_-^\eps$ is 
$\O(\eps^2\log|t|)$, for $t\le -10$: its contribution is negligible
for times $t$ such that $\eps^2|\log|t||\ll 1$, so in particular we can match with a linear
solution at $t^\eps_\gamma = -1/\eps^\gamma$,
\begin{equation*}
  i\d_t \u^\eps_1 +\frac{1}{2}\d_x^2 \u^\eps_1 =0 ,
\end{equation*}
with
\begin{equation*}
  \u^\eps_{1\mid
    t=-1/\eps^\gamma} = 
  \frac{1}{(it)^{1/2}}e^{i\frac{x^2}{2t}}\F(u_-^\eps)\(\frac{x}{t}\)
  \big|_{t=-1/\eps^\gamma}. 
\end{equation*}
The right hand side is $M(t)D(t)\F u_-^\eps$ evaluated at $t = -1/\eps^\gamma$. 
In view of Lemma~\ref{lem:R},  up to a small error term, 
\begin{equation*}
 \u_1^\eps(t) =
 U(t)u_-^\eps, \text{ or, equivalently, }\u_1^\eps=\eps v_1,
 \text{ with }v_1(t)=U(t)v_-.
\end{equation*}
We choose this definition for the first order approximation
$\u_1^\eps$. 
Let $t\ge -1/\eps^\gamma$. Duhamel's formula implies
\begin{align*}
  U(-t)\(u^\eps(t)- v^\eps_1(t)\) &= U\(-s\)\(
  u^\eps(s)- u^\eps_2(s)+u_2^\eps(s)
                                    -\u^\eps_1(s)\)\Big|_{s=-1/\eps^\gamma} \\
  &\quad-i\lambda
  \int_{-1/\eps^\gamma}^t U(-\tau)\(|u^\eps|^2u^\eps\)(\tau)d\tau,
\end{align*}
that is
\begin{align*}
  U(-t)u^\eps(t)- u_-^\eps &= U\(-s\)\(
  u^\eps(s)- u^\eps_2(s)\)\Big|_{s=-1/\eps^\gamma}
                             +U(-s)u_2^\eps(s)\Big|_{s=-1/\eps^\gamma}
                           -  u_-^\eps\\ 
  &\quad-i\lambda
  \int_{-1/\eps^\gamma}^t U(-\tau)\(|u^\eps|^2u^\eps\)(\tau)d\tau,
\end{align*}
The first term on the right hand side is estimated thanks to
Proposition~\ref{prop:wave-op}:
\begin{equation*}
  \left\|U\(-s\)\(
  u^\eps(s)- u^\eps_2(s)\)\Big|_{s=-1/\eps^\gamma}\right\|_{\Sigma}\lesssim
  \eps^{3+\gamma}|\log\eps|.
\end{equation*}
On the other hand,
\begin{equation*}
  \left\|U(-s)u_2^\eps(s)\Big|_{s=-1/\eps^\gamma}
                           -  u_-^\eps\right\|_{\Sigma} =
                         \|\F^{-1}\widehat w_{\mid s=-1/\eps^\gamma}- u_-^\eps\|_{\Sigma}= \|\widehat
                         w_{\mid s=-1/\eps^\gamma}- \F( u_-^\eps)\|_{\Sigma}.
\end{equation*}
Recalling that
\begin{equation*}
  \widehat w =  \F( u_-^\eps) \exp \(i|\F( u_-^\eps) |^2\log|t|\), 
\end{equation*}
we have directly
\begin{align*}
  \|\widehat w- \F( u_-^\eps)\|_{\Sigma}&\lesssim \|\<x\> \(\widehat w- \F(
  u_-^\eps)\)\|_{L^2} + \left\|\d_x \widehat{u_-^\eps}\( \exp \(i|\F( u_-^\eps)
                                      |^2\log|t|\)-1\)\right\|_{L^2} \\
  &\quad + |\log\eps|\left\| \F( u_-^\eps)\d_x |\F( u_-^\eps)
    |^2\right\|_{L^2}\\
  &\lesssim \eps^3|\log \eps|. 
\end{align*}
The integral term can be estimated in $\Sigma$ by
\begin{equation*}
  \int_{-1/\eps^\gamma}^t
  \|u^\eps(\tau)\|_{L^\infty}^2\|U(-\tau)u^\eps(\tau)\|_{\Sigma}d\tau.
\end{equation*}
Since $\| U(T)u^\eps(-T)\|_{\Sigma}\lesssim \eps$
from Lemmas~\ref{lem:u1} and \ref{lem:Rw}, and
Proposition~\ref{prop:wave-op} (where $T$ 
is independent of $\eps$), we have the following uniform estimate from
\cite[Theorem~1.1]{HN98} (see also \cite{KatoPusateri2011}), provided
that $\eps>0$ 
is sufficiently small,
\begin{equation}\label{eq:unifLinfty}
  \|u^\eps(\tau)\|_{L^\infty}\lesssim
  \frac{\eps}{\<\tau\>^{1/2}},\quad \forall \tau\in \R. 
\end{equation}
Writing
\begin{equation*}
  \|U(-\tau)u^\eps(\tau)\|_{\Sigma}\le
  \|U(-\tau)\(u^\eps(\tau)-\u_1^\eps(\tau)\)\|_{\Sigma} +
  \|U(-\tau)\u_1^\eps(\tau)\|_{\Sigma}, 
\end{equation*}
and using the obvious fact that
$\|U(-\tau)\u_1^\eps(\tau)\|_{\Sigma}=\eps\|v_-\|_{\Sigma}$, we have, for $t\ge
-1/\eps^\gamma$, 
\begin{equation*}
  \left\|   U(-t)u^\eps(t)- u_-^\eps\right\|_{\Sigma}\lesssim
  \eps^{3}|\log \eps| + \eps^2 \int_{-1/\eps^\gamma}^t \left\|
    U(-\tau)u^\eps(\tau)- u_-^\eps\right\|_{\Sigma}
  \frac{d\tau}{\<\tau\>} +  \eps^3 \int_{-1/\eps^\gamma}^t \frac{d\tau}{\<\tau\>}.
\end{equation*}
 Therefore, Gronwall
 lemma yields
 \begin{equation*}
   \sup_{-1/\eps^\gamma\le t\le T}\left\|   U(-t)u^\eps(t)-
     u_-^\eps\right\|_{\Sigma} \lesssim 
    \eps^{3}|\log \eps| e^{ C \eps^2|\log \eps|+C\eps^2\log\<T\>},
  \end{equation*}
  for some constant $C$ independent of $\eps$ and $T$. In particular,
  for any $\delta\in ]0,1]$, 
  the right hand side is $\O(\eps^{3-\delta})$ on
  $[-1/\eps^\gamma,1/\eps^\beta]$ for any $\beta>0$. This estimate
  implies the property
  \begin{equation*}
    \W (\eps v_-)=\eps v_- +\O(\eps^{3-\eta}),
  \end{equation*}
  for any $\eta>0$. In the next subsection, we improve this estimate
  by describing the first corrector term.  
\subsection{Next term in the asymptotic expansion}
\label{sec:next1}

Following e.g. \cite{CaGa09}, a refined asymptotic expansion for
$u^\eps$, at least on bounded time intervals, is given by the solution
$\u_2^\eps$ to
\begin{equation*}
  i\d_t \u_2^\eps +\frac{1}{2}\d_x^2 \u_2^\eps = \lambda
  |\u_1^\eps|^2\u_1^\eps, 
\end{equation*}
where the Cauchy data must be chosen carefully.
Formally, as $\u_1^\eps$ is
  of order $\eps$, $\u_2^\eps$ is of order $\eps^3$, provided that its
  Cauchy datum is $\O(\eps^3)$.
 Set
\begin{equation*}
  \u_{\rm app}^\eps = \u_1^\eps+\u^\eps_2.
\end{equation*}
We now have
\begin{align*}
  U(-t)\(u^\eps(t)- \u^\eps_{\rm app}(t)\) &= U\(-s\)\(
  u^\eps(s)- u^\eps_2(s)+u_2^\eps(s)
  - \u_{\rm app}^\eps(s)\)\Big|_{s=-1/\eps^\gamma} \\
  &\quad-i\lambda
  \int_{-1/\eps^\gamma}^t U(-\tau)\(|u^\eps|^2u^\eps-
    |\u_1^\eps|^2\u_1^\eps\)(\tau)d\tau.
\end{align*}
We have seen  before
that
\begin{equation*}
  \left\|U\(-s\)\(
  u^\eps(s)- u^\eps_2(s)\)\Big|_{s=-1/\eps^\gamma}\right\|_{\Sigma}\lesssim
  \eps^{3+\gamma}|\log\eps|=o(\eps^3). 
\end{equation*}
So if we want to catch the $\eps^3$ term, this is good. Making $U(-s)(u_2^\eps(s)
 -   \u_{\rm app}^\eps(s))$ small ($o(\eps^3)$ in $\Sigma$) at
 $s=-1/\eps^\gamma$ is what should 
 tell us how to choose the data for $\u_2^\eps$. Indeed, the integral
 term is estimated in $L^2$ ($\Sigma$ will require more care) by
 \begin{equation*}
     \int_{-1/\eps^\gamma}^t \(\|u^\eps(\tau)\|_{L^\infty}^2 +
     \|\u_1^\eps(\tau)\|_{L^\infty}^2
     \)\|u^\eps(\tau)-v_1^\eps(\tau)\|_{L^2}d\tau\lesssim \eps^2
     \int_{-1/\eps^\gamma}^t
     \|u^\eps(\tau)-\u_1^\eps(\tau)\|_{L^2}\frac{d\tau}{\<\tau\>}, 
   \end{equation*}
   and we now write
   \begin{equation*}
     \|u^\eps(\tau)-\u_1^\eps(\tau)\|_{L^2}\le \|u^\eps(\tau)-\u_{\rm
       app}^\eps(\tau)\|_{L^2}+\|\u_2^\eps(\tau)\|_{L^2}, 
   \end{equation*}
so the last term yields a smaller contribution (in terms of $\eps$ at
least) than in the previous
subsection.

Thus, everything seems to boil down to choosing $\u_2^\eps$ at time
$-1/\eps^\gamma$ in an efficient way, so that
\begin{equation*}
  \left\|U(-s)(u_2^\eps(s)-\u_{\rm
      app}^\eps(s))\Big|_{s=-1/\eps^\gamma}\right\|_\Sigma=o(\eps^3). 
\end{equation*}
By construction,
\begin{equation*}
  U(-s)(u_2^\eps(s))\Big|_{s=-1/\eps^\gamma}= \F^{-1}\( \F(u_-^\eps)
  \exp\(-i\lambda\gamma |\F(u_-^\eps)|^2\log \eps\)\). 
\end{equation*}
Taking into account the second term in the asymptotic expansion of the
exponential in $\widehat w$, we define $\u_2^\eps$ by
\begin{equation}\label{eq:v2}
  \begin{aligned}
 U(-t) \u_2^\eps (t) &=   -i\lambda\gamma \eps^3(\log
  \eps)\F^{-1}\( 
  |\F(v_-)|^2\F(v_-) \)\\
  &\quad-i\lambda\eps^3\int_{-1/\eps^\gamma}^t U(-\tau) \(
  |U(\tau)v_-|^2U(\tau)v_-\)d\tau. 
\end{aligned}
\end{equation}
\begin{lemma}\label{lem:renormalisation}
  Let $v_-\in \H$. For every $t\in \R$, $U(-t)\u^\eps_2(t)/\eps^3$ converges in $L^2\cap L^\infty(\R)$, and the
  limit is independent of $\gamma$.  This limit is given by
  \begin{align*}
   U(-t) v_2(t) & = -i\lambda\int_{-\infty}^{-1} \(  U(-\tau)\(
      |U(\tau)v_-|^2U(\tau)v_-\) +\frac{1}{|\tau|} \F^{-1}\( 
      |\F(v_-)|^2\F(v_-)\)\) d\tau\\
&\quad  -i\lambda\int_{-1}^t U(-\tau) \( 
 |U(\tau)v_-|^2U(\tau)v_-\)d\tau.
  \end{align*}
   The limit also holds in $\Sigma$,
  and for all $t\in \R$,
  \begin{equation*}
    \left\|
      U(-t)\(\frac{\u^\eps_2(t)}{\eps^3}-v_2(t)\)\right\|_{\Sigma}\lesssim
    \eps^\gamma.
  \end{equation*}
\end{lemma}
\begin{remark}
  We emphasize that even though the limit $v_2$ is independent of
$\gamma$, the error estimate improves for large values of $\gamma$. In
view of the definition of $\u_2^\eps$, the matching condition at
$t=-1/\eps^\gamma$ 
between $u_2^\eps$ and $\u_{\rm app}^\eps$ is not better than
$\O\(\eps^5(\log\eps)^2\)$, so it is no use (in this paper) to consider
$\gamma>2$. 
\end{remark}
\begin{proof}
  The proof consists in writing a precise asymptotic expansion of the
  argument of the integral involved in \eqref{eq:v2}. Writing, for any
  $\theta\in ]0,1]$,
  \begin{equation*}
    M(t) = e^{i\frac{x^2}{2t}} = 1+\O\(\left|\frac{x^2}{t}\right|^\theta\) ,
  \end{equation*}
  we have, in the same vein as in Lemma~\ref{lem:R},
  \begin{equation*}
    U(\tau)v_-(x) = M(\tau)D(\tau)\F v_- + R_1(\tau,x),
  \end{equation*}
  where, as $\tau\to -\infty$
  \begin{equation*}
    \|R_1(\tau)\|_{L^2} = \O\(\frac{1}{|\tau|^{\theta}}\),\quad
    \|R_1(\tau)\|_{L^\infty} = \O\(\frac{1}{|\tau|^{\theta+1/2}}\).
  \end{equation*}
  We infer, as $\tau\to -\infty$,
  \begin{equation*}
    |U(\tau)v_-|^2U(\tau)v_-=
    e^{i\frac{\pi}{4}}\frac{e^{i\frac{x^2}{2\tau}}}{|\tau|^{3/2}}
    \(|\F(v_-)|^2\F(v_-)\)\(\frac{x}{\tau}\)
    + R_2(\tau,x),
  \end{equation*}
  with
  \begin{equation*}
     \|R_2(\tau)\|_{L^2} = \O\(\frac{1}{|\tau|^{1+\theta}}\),\quad
    \|R_2(\tau)\|_{L^\infty} = \O\(\frac{1}{|\tau|^{3/2+\theta}}\).
  \end{equation*}
  In view of \eqref{eq:U(-t)}, we next compute
  \begin{align*}
   U(-\tau)\(  |U(\tau)v_-|^2U(\tau)v_-\) &= ie^{-i\frac{x^2}{2\tau}}
   \F^{-1}D\(\frac{1}{\tau}\) \( \frac{ e^{i\frac{\pi}{4}}}{|\tau|^{3/2}}
                                            \(|\F(v_-)|^2\F(v_-)\)\(\frac{x}{\tau}\)\)\\
    &\quad
    +ie^{-i\frac{x^2}{2\tau}}
   \F^{-1}D\(\frac{1}{\tau}\) \(  e^{-i\frac{x^2}{2\tau}}
      R_2(\tau,x)\)\\
    &= -\frac{e^{-i\frac{x^2}{2\tau}}}{|\tau|} \F^{-1}\(
      |\F(v_-)|^2\F(v_-)\)(x) + R_3(\tau,x)\\
   &= -\frac{1}{|\tau|} \F^{-1}\(
      |\F(v_-)|^2\F(v_-)\)(x) + R_4(\tau,x),
  \end{align*}
  where
\begin{equation*}
  \|R_4(\tau)\|_{L^2\cap L^\infty} =
  \O\(\frac{1}{|\tau|^{1+\theta}}\).
\end{equation*}
Set $\theta=1$,
and write
\begin{align*}
  \frac{1}{\lambda \eps^3} U(-t)\u_2^\eps(t) &= -i \gamma
  (\log\eps)\F^{-1}  \( |\F(v_-)|^2\F(v_-) \)
  -i\int_{-1/\eps^\gamma}^{-1} U(-\tau) \(
    |U(\tau)v_-|^2U(\tau)v_-\)d\tau\\
&\quad-i\int_{-1}^t U(-\tau) \(
 |U(\tau)v_-|^2U(\tau)v_-\)d\tau\\
 &= -i\gamma
  (\log\eps)\F^{-1}  \( |\F(v_-)|^2\F(v_-) \)\\
  &\quad-i\int_{-1/\eps^\gamma}^{-1} \(-\frac{1}{|\tau|} \F^{-1}\(
      |\F(v_-)|^2\F(v_-)\) + R_4(\tau)\)d\tau\\
&\quad-i\int_{-1}^t U(-\tau) \(
 |U(\tau)v_-|^2U(\tau)v_-\)d\tau\\
 &= -i\int_{-\infty}^{-1} R_4(\tau)d\tau +\underbrace{ \O\(
   \int_{-\infty}^{-1/\eps^\gamma} \frac{d\tau}{\tau^2}\)}_{=\O(\eps^\gamma)}
-i\int_{-1}^t U(-\tau) \(
 |U(\tau)v_-|^2U(\tau)v_-\)d\tau.
\end{align*}
Therefore,
\begin{equation*}
  U(-t)v_2(t) = -i\lambda\int_{-\infty}^{-1} R_4(\tau)d\tau
  -i\lambda\int_{-1}^t U(-\tau) \( 
 |U(\tau)v_-|^2U(\tau)v_-\)d\tau,
\end{equation*}
and, by construction,
\begin{equation*}
   R_4(\tau) =  U(-\tau)\(  |U(\tau)v_-|^2U(\tau)v_-\) +\frac{1}{|\tau|} \F^{-1}\(
      |\F(v_-)|^2\F(v_-)\),
    \end{equation*}
    hence the formula of the lemma.
In particular, by the same arguments as above, the second integral
diverges logarithmically as $t\to +\infty$. We leave out the convergence in
$\Sigma$ (since $v_-\in \H$, so we can pay some 
momentum estimate when controlling $M-1$),  which relies on similar ideas,
like the proof of 
Lemma~\ref{lem:R}. Note that the present argument 
\end{proof}
We can now improve the error estimate proven in the previous
subsection with only the leading order approximation $\u_1^\eps$. By
construction,
\begin{equation*}
   \left\|U(-s)(u_2^\eps(s)-\u_{\rm
      app}^\eps(s))\Big|_{s=-1/\eps^\gamma}\right\|_\Sigma=\O\(\eps^5(\log \eps)^2\),
\end{equation*}
and we have, if $0<\gamma<2$,
\begin{align*}
  \|U(-t)\(u^\eps(t)- \u^\eps_{\rm app}(t)\)\|_{\Sigma} &\lesssim
  \eps^{3+\gamma}|\log \eps|\\
  &\quad+
  \int_{-1/\eps^\gamma}^t \| U(-\tau)\(|u^\eps|^2u^\eps-
    |\u_1^\eps|^2\u_1^\eps\)(\tau)\|_{\Sigma}d\tau.
\end{align*}
In view of \eqref{eq:est-gen},
\begin{align*}
   \| U(-\tau)\(|u^\eps|^2u^\eps-
    |\u_1^\eps|^2\u_1^\eps\)(\tau)\|_{\Sigma} \lesssim
  \| &u^\eps(\tau)\|_{L^\infty}^2 \| U(-\tau)\(u^\eps-
  \u_1^\eps\)(\tau)\|_{\Sigma} \\
  + \( \|u^\eps(\tau)\|_{L^\infty} +
  \|\u_1^\eps(\tau)\|_{L^\infty} \) &\|U(-\tau)\u_1^\eps(\tau)\|_{\Sigma}
  \|u^\eps(\tau)-\u_1^\eps(\tau)\|_{L^\infty}. 
\end{align*}
Recalling \eqref{eq:unifLinfty}, we readily have
\begin{equation*}
  \|u^\eps(\tau)\|_{L^\infty} \lesssim \frac{\eps}{\<\tau\>^{1/2}}.
\end{equation*}
On the other hand,
\begin{equation*}
  \|U(-\tau)\u_1^\eps(\tau)\|_{\Sigma}=\eps\|v_-\|_{\Sigma},
\end{equation*}
and Sobolev embedding yields, along with \eqref{eq:GNlibre},
\begin{equation*}
  \|\u_1^\eps(\tau)\|_{L^\infty} \lesssim \frac{\eps}{\<\tau\>^{1/2}},
\end{equation*}
as well as
\begin{equation*}
  \|u^\eps(\tau)-\u_1^\eps(\tau)\|_{L^\infty}\lesssim
  \frac{1}{\<\tau\>^{1/2}}\|U(-\tau)\(u^\eps-
  v_1^\eps\)(\tau)\|_{\Sigma} . 
\end{equation*}
Therefore,
\begin{align*}
   \| U(-\tau)\(|u^\eps|^2u^\eps-
    |\u_1^\eps|^2\u_1^\eps\)(\tau)\|_{\Sigma} &\lesssim
    \frac{\eps^2}{\<\tau\>} \|U(-\tau)\(u^\eps-
     \u_1^\eps\)(\tau)\|_{\Sigma}\\
   \lesssim
    \frac{\eps^2}{\<\tau\>} & \(\|U(-\tau)\(u^\eps-
    \u_{\rm app}^\eps\)(\tau)\|_{\Sigma} +
    \|U(-\tau)\u_2^\eps(\tau)\|_{\Sigma}\). 
\end{align*}
We see that
there exists $C$ such that for all $T\ge 1$,
\begin{equation*}
   \|U(-\tau)\u_2^\eps(\tau)\|_{\Sigma}\lesssim \eps^3\log\<T\>,\quad
   \forall \tau \in\left[-\frac{1}{\eps^\gamma},T\right],
 \end{equation*}
 where the logarithmic correction is necessarily present for large
 positive $\tau$, as pointed out above. Gronwall lemma now yields, for
 all $0<\gamma<2$, 
 \begin{equation*}
   \sup_{-1/\eps^\gamma\le t\le T}\left\|   U(-t)\(u^\eps(t)-
     \u_{\rm app}^\eps(t)\)\right\|_{\Sigma} \lesssim 
    \eps^{3+\gamma}|\log \eps| e^{ C \eps^2|\log \eps|+C\eps^2\log\<T\>},
  \end{equation*}
  for some constant $C$ independent of $\eps$ and $T$. In particular,
  the right hand side is $o(\eps^{3})$ on
  $[-1/\eps^\gamma,1/\eps^\beta]$ for any $\beta>0$. 

  The first point of Theorem~\ref{theo:main} follows, by considering
  the above error estimate at time $t=0$ and setting $w_2
  = v_{2\mid t=0}$.

\section{Asymptotic completeness}
\label{sec:complete}

\subsection{Main steps of the construction of $u_+$}

We recall the main steps from the proof of \cite[Theorem~1.2]{HN98}, a
result which we state in the particular case that we shall consider
(the initial data may have a different regularity there, and an extra
short range nonlinearity can be incorporated). In passing, we keep
track of the size of the remainders more precisely.
\begin{theorem}[From \cite{HN98}]\label{theo:complete}
  Let $u_0\in \Sigma$, with $\|u_0\|_{\Sigma}=\eps'\le \delta$, where
  $\delta>0$ is sufficiently small. There exist unique functions $W\in
  L^\infty\cap L^2$ and $\Phi\in L^\infty$ such that, for $t\ge 1$ and $C\delta<\alpha<1/4$, 
  \begin{equation}
    \label{eq:1.4}
     \left\| \F\(U(-t)u(t)\) \exp\(i\lambda\int_1^t|\widehat
     u(\tau)|^2\frac{d\tau}{\tau}\)-W\right\|_{L^2\cap L^\infty} 
     \le C(\eps')^3  t^{-\alpha+C(\eps')^2}, 
  \end{equation}
and
\begin{equation}
    \label{eq:1.5}
      \left\| \lambda \int_1^t |\widehat u(\tau)|^2\frac{d\tau}{\tau}
     -\lambda |W|^2 \log t-\Phi\right\|_{L^\infty} 
     \le C(\eps')^3  t^{-\alpha+C(\eps')^2}.
   \end{equation}
  In particular,
  \begin{equation*}
    u(t,x) =\frac{1}{(it)^{1/2}}W\(\frac{x}{t}\)\exp
    \(i\frac{x^2}{2t}-i\lambda \left|W\(\frac{x}{t}\)\right|^2\log t -
    i\Phi\(\frac{x}{t}\)\) + \rho(t,x),
  \end{equation*}
  with
  \begin{equation*}
    \|\rho(t)\|_{L^2} \le C\eps' t^{-\alpha +C(\eps')^2},\quad
      \|\rho(t)\|_{L^\infty} \le C\eps' t^{-1/2-\alpha +C(\eps')^2} .
  \end{equation*}
\end{theorem}
In \cite{HN98}, the time decay in \eqref{eq:1.5}  is raised to the
power $2/3$, because the setting is more general and includes the case
of dimension three. We explain below why this power can be discarded
in the above statement. The asymptotic state $u_+$ is naturally given
by
\begin{equation*}
  u_+ = \F^{-1} \(We^{-i\Phi}\).
\end{equation*}
\begin{proof}[Sketch of the proof]
  As announced above, we recall the main steps from \cite{HN98}, and
  slightly improve some estimates in terms of the powers of
  $\eps'$.

As evoked before in the case of \eqref{eq:unifLinfty}, (the proof of) \cite[Theorem~1.1]{HN98} provides global estimates
for $u$ (proven in \cite[Lemma~3.2 and 3.3]{HN98}, see also \cite{KatoPusateri2011}):
\begin{equation}\label{eq:unif-u}
  \|u(t)\|_{L^\infty}\lesssim
    \frac{\eps'}{\<t\>^{1/2}},\quad
   \|U(-t)u(t)\|_{\Sigma}\lesssim \eps' \<t\>^{C(\eps')^2}.
\end{equation}
Denote
  \begin{equation*}
    v(t)=U(-t)u(t),\quad \widehat w = \widehat v \underbrace{\exp\(i\lambda
        \int_1^t |\widehat v(\tau)|^2\frac{d\tau}{\tau}\)}_{=:B(t)}.
    \end{equation*}
 Then \eqref{eq:NLS} is recasted as
 \begin{equation}\label{eq:w-complete}
   i\d_t \widehat w =\frac{\lambda}{t}B(t)\(I_1+I_2\),
 \end{equation}
 where
 \begin{equation*}
   I_1(t)= \F\(M(-t)-1\)\F^{-1}\(|\widehat{M(t) v}|^2\widehat{M(t) v}\),\quad
  I_2(t)= \left|\widehat{M(t)v}\right|^2 \widehat{M(t)v}- |\widehat
  v|^2\widehat v.
\end{equation*}
These source terms are controlled as follows, for $t\ge 1$:
\begin{equation*}
  \|I_1(t)\|_{L^2\cap L^\infty} + \|I_2(t)\|_{L^2\cap L^\infty} \lesssim \frac{\|v\|_{\Sigma}^3}{t^\alpha},
\end{equation*}
provided that $\alpha<1/2$. In view of \eqref{eq:unif-u} and the
definition of $v$, this entails
\begin{equation*}
   \|I_1(t)\|_{L^2\cap L^\infty} + \|I_2(t)\|_{L^2\cap L^\infty}
   \lesssim (\eps')^3 t^{-\alpha+3C(\eps')^2}. 
 \end{equation*}
Renaming $3C$ to $C$, and integrating \eqref{eq:w-complete}, we
readily find that there exists 
$W\in L^2\cap L^\infty$ such that
\begin{equation}\label{eq:cvw}
  \|\widehat w(t)-W\|_{L^2\cap L^\infty} \lesssim (\eps')^3 t^{-\alpha+C(\eps')^2},
\end{equation}
thus proving \eqref{eq:1.4}, since
\begin{equation*}
  |\widehat v(t,\xi) | = \left|\F\(U(-t)u(t)\)(\xi)\right|=
    \left|e^{-it\frac{\xi^2}{2}} \widehat u(t,\xi)\right|=|\widehat
    u(t,\xi)|.
\end{equation*}
The phase corrector $\Phi$ is obtained by introducing the function
\begin{equation}\label{eq:defPsi}
  \Psi (t) =\lambda\int_1^t \(|\widehat w(\tau)|^2-|\widehat
  w(t)|^2\)\frac{d\tau}{\tau}. 
\end{equation}
We have, for $t>s>2$,
\begin{equation*}
  \Psi(t)-\Psi(s) = \lambda\int_s^t \(|\widehat w(\tau)|^2-|\widehat
  w(t)|^2\)\frac{d\tau}{\tau}+\lambda \( |\widehat w(t)|^2-|\widehat
  w(s)|^2\)\log s.
\end{equation*}
In view of the above estimates, for $t_2>t_1>2$,
\begin{align*}
  \left| |\widehat w(t_2)|^2-|\widehat
  w(t_1)|^2\right|&\lesssim \( |\widehat w(t_2)|+|\widehat
  w(t_1)|\) |\widehat w(t_2)-\widehat
                    w(t_1)|\\
  &\lesssim \( |\widehat w(t_2)|+|\widehat
  w(t_1)|\)(\eps')^3 t_1^{-\alpha+C(\eps')^2}.
\end{align*}
On the other hand, we have
\begin{align*}
  |\widehat w(t)|\lesssim \eps',
\end{align*}
from \eqref{eq:cvw}. Therefore,
\begin{equation}\label{eq:diff-hat-w}
  \left| |\widehat w(t_2)|^2-|\widehat
  w(t_1)|^2\right|
  \lesssim (\eps')^4 t_1^{-\alpha+C(\eps')^2},\quad t_2>t_1>2,
\end{equation}
and so
\begin{equation*}
  |\Psi(t)-\Psi(s) |\lesssim (\eps')^4 s^{-\alpha+C(\eps')^2}\log s.
\end{equation*}
In particular, there exists $\Phi\in L^\infty$ such that, for $t>2$,
\begin{equation}\label{eq:cvPhi}
  |\Phi-\Psi(t) |\lesssim (\eps')^4 t^{-\alpha+C(\eps')^2}\log t.
\end{equation}
The estimate \eqref{eq:1.5} then follows from the identity
\begin{equation*}
  \lambda \int_1^t |\widehat w(\tau)|^2\frac{d\tau}{\tau} =
  \lambda|W|^2\log t + \Phi + \Psi(t)-\Phi+ \lambda\( |\widehat
  w(t)|^2-|W|^2\)\log t,
\end{equation*}
using \eqref{eq:cvw} and \eqref{eq:cvPhi}, possibly modifying the
constant $C$, and using the fact that the condition on $\alpha$ in
Theorem~\ref{theo:complete} is open. 

In view of \eqref{eq:1.4} and \eqref{eq:1.5}, we find
\begin{equation*}
  \F\(U(-t)u(t)\) = W e^{-i\lambda |W|^2\log t -i\Phi} +r(t),
\end{equation*}
where
\begin{equation*}
  \|r(t)\|_{L^2\cap
    L^\infty} \lesssim (\eps')^3 t^{-\alpha +C(\eps')^2}. 
\end{equation*}
Writing 
\begin{equation*}
  1=U(t)U(-t)= M(t)D(t)\F M(t) U(-t) = M(t)D(t)\F  U(-t) + \cR(t) \F
  U(-t) ,
\end{equation*}
we infer
\begin{equation*}
  u(t,x) =M(t)D(t)\F  U(-t) u(t)+\rho_1(t,x),
\end{equation*}
with, in view of Lemma~\ref{lem:R},
\begin{equation*}
  \|\rho_1(t) \|_{L^2} \lesssim \frac{\eps'}{\sqrt t}\|\F
  U(-t) u(t)\|_{L^2} \lesssim \frac{\eps'}{\sqrt t},
\end{equation*}
and, since $\alpha<1/4$, using the first point of Lemma~\ref{lem:R}
with $\theta=\alpha$ and $s=1-2\alpha$, 
\begin{equation*} 
\|\rho_1(t) \|_{L^\infty} \lesssim \frac{\eps'}{t^{1/2+\alpha}}\|\F
  U(-t) u(t)\|_{H^1}\lesssim  \eps' t^{-1/2-\alpha+C(\eps')^2}.
\end{equation*}
We conclude thanks to the estimates
\begin{align*}
  &\|M(t)D(t)r(t)\|_{L^2}=\|r(t)\|_{L^2} \lesssim (\eps')^3 t^{-\alpha
    +C(\eps')^2},\\
&\|M(t)D(t)r(t)\|_{L^\infty}=\frac{1}{\sqrt
    t}\|r(t)\|_{L^\infty}\lesssim (\eps')^3 t^{-1/2-\alpha
    +C(\eps')^2},
\end{align*}
since $u(t)= M(t)D(t) \( W e^{-i\lambda |W|^2\log t -i\Phi}\) +
M(t)D(t)r(t)+\rho_1(t)$. 
\end{proof}

\subsection{Modified scattering operator: leading order
  behavior near the origin}\label{sec:complete1}

In view of Theorem~\ref{theo:complete}, the asymptotic state that we
consider is $u_+^\eps= \F^{-1}\(W^\eps e^{-i\Phi^\eps}\)$, with
$u_0^\eps = u^\eps_{\mid t=0}$. The main remark at this stage is that
the reduction presented in the proof of Theorem~\ref{theo:complete}
boils down the analysis of the asymptotic behavior of $u_+^\eps$ as
$\eps\to 0$ to a regular asymptotic expansion. In order to treat the
last two cases of Theorem~\ref{theo:main}, we suppose that, for $0<\eta<2$,
\begin{equation*}
  u_0^\eps = \eps v_0 + \eps^3 w_2 + \O\(\eps^{5-\eta} \),
\end{equation*}
with $v_0,w_2\in \Sigma$ and the remainder term is (measured) in
$\Sigma$. For the last point of Theorem~\ref{theo:main}, we assume 
$w_2=\rho^\eps=0$, while for the second point, $v_0=v_-$ and $w_2 =
v_{2\mid t=0}$ as in the first point.

We keep the same notations as the proof of
Theorem~\ref{theo:complete}, with now $\eps$ instead  of $\eps'$. The
definition of $\widehat w$, in particular the term $B$, requires as a
first step the asymptotic description of $u^\eps$ at time $t=1$ instead
of only $t=0$.

In view of \cite[Theorem~1.1]{HN98}, as in \eqref{eq:unif-u}, we have
\begin{equation*}
   \|u^\eps(t)\|_{L^\infty}\lesssim
    \frac{\eps}{\<t\>^{1/2}},\quad
   \|U(-t)u^\eps(t)\|_{\Sigma}\lesssim \eps \<t\>^{C\eps^2}.
 \end{equation*}
 In view of the proof of \eqref{eq:1.4} and \eqref{eq:1.5}, we
 directly know that
 \begin{equation*}
   \|\d_t \widehat w\|_{L^1([1,\infty[,L^2\cap L^\infty)}\lesssim \eps^3,
 \end{equation*}
 and so
 \begin{equation*}
   \widehat w(t) = \widehat w(1) + \O(\eps^3) \text{ in }L^\infty([1,\infty[,L^2\cap L^\infty).
 \end{equation*}
 By definition, $\widehat w(1) = \widehat v(1) = \F\( U(-1)u^\eps(1)\)= \eps
 \widehat v_0+\O(\eps^3)$ in $\Sigma$.

 Regarding $\Phi$, the definition~\eqref{eq:defPsi}, and \eqref{eq:diff-hat-w},
 yield
 \begin{equation*}
   \|\Phi\|_{L^\infty}\le
   \|\Phi-\Psi(t)\|_{L^\infty}+\|\Psi(t)\|_{L^\infty}\lesssim \eps^4 ,
 \end{equation*}
where the right hand side is estimated uniformly in $t\ge 1$. 
 Therefore,
 \begin{equation*}
   u_+^\eps = \F^{-1}\(We^{-i\Phi}\) = \eps v_0+\O(\eps^3)\text{ in
   }L^2. 
 \end{equation*}
 Therefore, at leading order, we have
 \begin{equation*}
   \Scatt(\eps v_-) = \eps v_-+\O(\eps^3),\quad \(W_+^{\rm mod}\)^{-1}
   (\eps v_0) = \eps v_ 0 +\O(\eps^3)\quad \text{in }L^2.
 \end{equation*}

\section{Higher order asymptotic expansion of the final state}
\label{sec:complete-higher}

The higher order asymptotic expansion for $u_+^\eps$, involving an
 $\eps^3$ term, requires more work, even at the formal level. We first
 show how to derive this term, then, in a final subsection, prove the
 error estimates announced in Theorem~\ref{theo:main}.
 
 \subsection{Derivation of the first corrector}

We first examine the value of $u^\eps$ at time $t=1$, in view of
future connections with 
Theorem~\ref{theo:complete}.
The idea is the same as in Section~\ref{sec:next1}: for finite time,
since we consider small data, the asymptotic expansion is given by
Picard's scheme, and we have directly from Section~\ref{sec:wave-zero},
\begin{equation*}
  u^\eps_{\mid t=1} = \eps U(1)v_0 + \eps^3 \widetilde w_2  + \eps^{5-\eta} \widehat\rho^\eps,
\end{equation*}
with
\begin{equation*}
  \widetilde w_2 = U(1)w_2 -i\lambda\int_0^1 U(1-s)\(|U(s)v_0|^2U(s)v_0\)ds,
\end{equation*}
and $\|\rho^\eps\|_{\Sigma}\lesssim 1$ as $\eps \to 0$. 
\smallbreak

To derive the first corrector, we plug the asymptotic expansion 
\begin{equation*}
  \widehat w(t) = \eps \widehat v_0 + \eps^3 \widehat\nu_2 (t)
  +\eps^{5-\eta}\widehat{r^\eps}(t)
\end{equation*}
into \eqref{eq:w-complete}, with $r^\eps=\O(1)$ in some topology we
shall precise. We first proceed formally, and then check
that the above ansatz indeed provides a corrector of $u^\eps_+$ in
$L^2$. We compute successively
\begin{equation*}
  |\widehat v|^2=|\widehat w|^2 = \eps^2|\widehat v_0|^2 +\O(\eps^4),
\end{equation*}
hence, for bounded $t$,
\begin{equation*}
  B(t) = e^{i\lambda \int_1^t |\widehat v(\tau)|^2\frac{d\tau}{\tau}} = 1+
  i\lambda \eps^2|\widehat v_0|^2\log t + \O(\eps^4),
\end{equation*}
and
\begin{equation*}
  \widehat v = \widehat w \overline B = \eps \widehat v_0 +\eps^3\( \widehat\nu_2
  -i\lambda |\widehat v_0 |^2\widehat v_0 \log t\)+\O(\eps^{5-\eta}).
\end{equation*}
We then expand the factor $I_1^\eps$ and $I_2^\eps$ from \eqref{eq:w-complete}:
\begin{align*}
  I_1^\eps(t)& = \eps^3
 \F\(M(-t)-1\)\F^{-1}\(\left|\widehat{M(t)v_0}\right|^2 
       \widehat{M(t)v_0} \)+\O(\eps^5),\\
 I_2^\eps(t) & = \eps^3\( \left|\widehat{M(t)v_0}\right|^2
       \widehat{M(t)v_0} - |\widehat v_0 |^2\widehat v_0 \)+\O(\eps^5).
\end{align*}
Therefore, the natural candidate for $\nu_2$ satisfies:
\begin{equation*}
  \d_t \widehat\nu_2 = \frac{\lambda}{t}(J_1+J_2),
\end{equation*}
where
\begin{align*}
   J_1& = \F\(M(-t)-1\)\F^{-1}\(\left|\widehat{M(t)v_0}\right|^2
       \widehat{M(t)v_0} \),\\
 J_2 & = \left|\widehat{M(t)v_0}\right|^2
       \widehat{M(t)v_0} - |\widehat v_0 |^2\widehat v_0 .
\end{align*}
The Cauchy datum for $\nu_2$ is given at $t=1$: on the one hand, the
above asymptotic expansion implies
\begin{equation*}
  \widehat v_{\mid t=1} =\eps \widehat v_0 +\eps^3\widehat{\nu}_{2\mid t=1}
  +\O(\eps^{5-\eta}),
\end{equation*}
while on the other hand, the analysis of Section~\ref{sec:wave-zero} yields
\begin{equation*}
  U(-t)u^\eps(t)_{\mid t=1} = \eps v_0 + \eps^3 U(-1)\widetilde w_2 + \O(\eps^{5-\eta}),
\end{equation*}
so we infer
\begin{equation*}
  \nu_{2\mid t=1}=   U(-1)\widetilde w_2  = w_2
  -i\lambda \int_0^1 U(-s)\( |U(s)v_0|^2U(s)v_0\)ds.
\end{equation*}
As $v_0\in \Sigma$, the proof of Theorem~\ref{theo:complete} readily
yields $\frac{1}{t}(J_1+J_2)\in L^1([1,\infty[;L^2\cap L^\infty)$.
Since $\Sigma$ is a Banach algebra, invariant under the Fourier
transform, and such that $\Sigma\hookrightarrow L^1\cap L^\infty$, we
have $\nu_{2\mid t=1}\in \Sigma$, and so $\nu_2\in
L^\infty([1,\infty[;L^2\cap L^\infty)$ and
\begin{equation*}
  \widehat\nu_2(t)\Tend t \infty \widehat\nu_{2\mid t=1} + \lambda \int_1^\infty
  (J_1(\tau)+J_2(\tau))\frac{d\tau}{\tau} \quad\text{in }L^2\cap
  L^\infty. 
\end{equation*}

However, in the asymptotic expansion of $B$,  $I_1$ and $I_2$,
the more precise information that we need is that the remainder term
is $\O(\eps^5)$ together with some algebraic decay in time, so we
obtain an asymptotic expansion for $\widehat w$ in $L^2$.  This requires
more effort, and the proof is presented in the next subsection.

\subsection{Justification of the second order asymptotic expansion}
As a first step, we analyze the regularity of the  corrector
$\nu_2$. To do so, we introduce an extra function space: for
$0<\gamma<1$, set
  \begin{equation*}
     \Sigma^\gamma = H^\gamma\cap \F(H^\gamma)=\{f\in H^\gamma(\R),\
     \|f\|_{\Sigma^\gamma}:= \|f\|_{H^\gamma} +\| \<x\>^\gamma f\|_{L^2}<\infty\}.
   \end{equation*}
  We note that like $\Sigma$, $\Sigma^\gamma$ is invariant under the
  Fourier transform, as well as the action of $U(t)$, and,  when
  $\gamma>1/2$,  it is an
  algebra embedded into $L^1\cap L^\infty$.
\begin{lemma}\label{lem:nu2}
   Let $v_0,w_2\in \Sigma$. Recall that the first corrector is defined by
  \begin{equation*}
     \widehat\nu_2(t) =  \widehat w_2
  -i\lambda \int_0^1 \F U(-s)\( |U(s)v_0|^2U(s)v_0\)ds +\lambda\int_1^t
                       (J_1(\tau)+J_2(\tau))\frac{d\tau}{\tau} , 
  \end{equation*}
  where $J_1$ and $J_2$ are defined by
  \begin{align*}
   J_1(t)& = \F\(M(-t)-1\)\F^{-1}\(\left|\widehat{M(t)v_0}\right|^2
       \widehat{M(t)v_0} \),\\
 J_2(t) & = \left|\widehat{M(t)v_0}\right|^2
       \widehat{M(t)v_0} - |\widehat v_0 |^2\widehat v_0 .
  \end{align*}
  Then $J_1,J_2\in L^\infty([1,\infty[,\Sigma)$, with 
  \begin{equation*}
    \sup_{t\ge 1}\|J_1(t)\|_\Sigma +  \sup_{t\ge
      1}\|J_2(t)\|_\Sigma\lesssim \|v_0\|_{\Sigma}^3,
  \end{equation*}
 and, for any $1/2<\gamma<1$,
  \begin{equation*}
    \|J_1(t)\|_{\Sigma^\gamma}+ \|J_2(t)\|_{\Sigma^\gamma}\lesssim
      \frac{\|v_0\|_{\Sigma}^3}{t^{(1-\gamma)/2}} ,
  \end{equation*}
 and therefore  $\nu_2\in
 L^\infty_{\rm loc}([1,\infty[,\Sigma)\cap
 L^\infty([1,\infty[,\Sigma^\gamma)$. Finally, $\widehat \nu_2(t)\to
 \widehat\nu_2^\infty$ in $\Sigma^\gamma$ as $t\to \infty$, where
 \begin{equation*}
   \widehat\nu_2^\infty = \widehat w_2
  -i\lambda \int_0^1 \F U(-s)\( |U(s)v_0|^2U(s)v_0\)ds +\lambda\int_1^\infty
                       (J_1(\tau)+J_2(\tau))\frac{d\tau}{\tau} .
 \end{equation*}
\end{lemma}
\begin{proof}[Proof of Lemma~\ref{lem:nu2}]
  Since $w_2,v_0\in \Sigma$, the first two terms defining $\widehat
  \nu_2$ belong to $\Sigma$ (as it is an algebra). For any (fixed)
  $t$, $J_1,J_2\in \Sigma$, and $\nu_2(t)\in \Sigma$: the uniform
  bound in time is straightforward, but to get some time decay, we pay
  a little regularity. Let $0<\gamma<1$: 
  like in the proof of Lemma~\ref{lem:R}, write, for $t\ge 1$,
  \begin{align*}
    \|J_1(t)\|_{\Sigma^\gamma} &= \left\|
    \(M(-t)-1\)\F^{-1}\(\left|\widehat{M(t)v_0}\right|^2 
                              \widehat{M(t)v_0} \)\right\|_{\Sigma^\gamma}\\
                            &\lesssim
    \frac{1}{t^{(1-\gamma)/2}}
    \left\|
    |x|^{1-\gamma}\F^{-1}\(\left|\widehat{M(t)v_0}\right|^2 
     \widehat{M(t)v_0} \)\right\|_{\Sigma^\gamma} \\
    &\lesssim \frac{1}{t^{(1-\gamma)/2}} \left\|
    \F^{-1}\(\left|\widehat{M(t)v_0}\right|^2 
     \widehat{M(t)v_0} \)\right\|_{\Sigma}\lesssim
      \frac{\|v_0\|_{\Sigma}^3}{t^{(1-\gamma)/2}} , 
  \end{align*}
  since $\Sigma$ is an algebra. The assumption
  $\gamma>1/2$ simplifies the computations in the case of $J_2$, as
  $\Sigma^\gamma$ is an algebra, and the first 
  inequality below is straightforward: 
  \begin{align*}
     \|J_2(t)\|_{\Sigma^\gamma} &\lesssim \(\|M(t)v_0\|_{\Sigma^\gamma}^2 +
                               \|v_0\|_{\Sigma^\gamma}^2\)
                               \|(M(t)-1)v_0\|_{\Sigma^\gamma}\\
    &\lesssim
    \|v_0\|_{\Sigma}^2\frac{1}{t^{(1-\gamma)/2}}\||x|^{1-\gamma}
    v_0\|_{\Sigma^\gamma} \lesssim
    \frac{\|v_0\|_{\Sigma}^3}{t^{(1-\gamma)/2}} . 
  \end{align*}
  The lemma follows, since the extra decay in time,
  $t^{(\gamma-1)/2}$, ensures the convergence of the integral in the
  last term defining $\widehat \nu_2$. 
\end{proof}
We recall two uniform estimates which will be of constant use in the
course of the proof:
\begin{equation}\label{eq:control-hat-v}
  \|\widehat v(t)\|_{L^\infty}+
  \left\|\widehat{M(t)v}\right\|_{L^\infty}\lesssim \eps.
\end{equation}
The first quantity is estimated in the proof of
Theorem~\ref{theo:complete}, thanks  to \eqref{eq:cvw}. The second one
is controlled thanks to \eqref{eq:unif-u}, since, using
\eqref{eq:U(-t)}, 
\begin{equation*}
  \left\|\widehat{M(t)v}\right\|_{L^\infty}= \left\|\F
 M(t)U(-t)u(t)\right\|_{L^\infty}
 =\left\|D\(\frac{1}{t}\)M(-t)u(t)\right\|_{L^\infty}
= \sqrt t  \left\|u(t)\right\|_{L^\infty}.
\end{equation*}

The justification of the formal asymptotic expansion derived in the
previous subsection relies on a Gronwall type argument. The error that
we want to control is $\widehat w -\eps\widehat v_0 -\eps^3\widehat
\nu_2$, which solves, by construction (keeping in mind that $v_0$ does
not depend on time),
\begin{equation*}
  \d_t\( \widehat w -\eps\widehat v_0 -\eps^3\widehat \nu_2\)=
  \frac{\lambda}{t}B(I_1^\eps+I_2^\eps) -\frac{\lambda}{t}\eps^3
  (J_1+J_2).  
\end{equation*}
We rewrite the right hand side as
\begin{equation*}
  \frac{\lambda}{t}B\(I_1^\eps+I_2^\eps-\eps^3J_1-\eps^3J_2\)
  +\frac{\lambda}{t}(B-1)\eps^3 
  (J_1+J_2), 
\end{equation*}
and the last term will be considered as a source term. Indeed, from
the definition
\begin{equation*}
  B(t) = e^{i\lambda \int_1^t |\widehat v(\tau)|^2\frac{d\tau}{\tau}}=
  e^{i\lambda \int_1^t |\widehat w(\tau)|^2\frac{d\tau}{\tau}} ,
\end{equation*} 
so from \eqref{eq:control-hat-v}, 
\begin{equation}
  \label{eq:B-1}
  \|B(t) -1\|_{L^\infty}\lesssim \eps^2\log t,
\end{equation}
and Lemma~\ref{lem:nu2} yields
\begin{equation}\label{eq:source-nu2-1}
  \left\|\eps^3(B-1)
  (J_1+J_2)\right\|_{L^2}\lesssim \eps^5 
\frac{\log t }{t^{(1-\gamma)/2}},
\end{equation}
for any  $1/2<\gamma<1$.

We now focus on the term $B(I_1^\eps+I_2^\eps-\eps^3J_1-\eps^3J_2)$,
to estimate it in $L^2$,
and examine successively $B(I_1^\eps-\eps^3J_1)$
and $B(I_2^\eps-\eps^3J_2)$. First, by definition,
\begin{equation*}
  I_1^\eps-\eps^3J_1 = \F\(M(-t)-1\)\F^{-1}\(
\left|\widehat{M(t)v}\right|^2
       \widehat{M(t)v}
  -\eps^3\left|\widehat{M(t)v_0}\right|^2
       \widehat{M(t)v_0} \),
\end{equation*}
so the term $v-\eps v_0$ is naturally factored out, and we must relate
it to the left hand side, involving $w-\eps v_0-\eps^3\nu_2$. Since
$B\overline B=1$, we can write
\begin{equation*}
 \widehat \rho := \widehat w -\eps \widehat v_0-\eps^3\widehat \nu_2 = B\widehat v
  -\eps \widehat v_0-\eps^3\widehat \nu_2  = B\(\widehat v
  -\eps \overline B \widehat v_0-\eps^3\overline B\widehat \nu_2\),
\end{equation*}
so we have
\begin{equation}
  \label{eq:recoll-DA}
  \widehat v-\eps\widehat v_0 = \overline B \( \widehat w -\eps
  \widehat v_0-\eps^3\widehat \nu_2\) +  \eps (\overline B -1)\widehat
  v_0 +\eps^3 \overline B \widehat \nu_2.
\end{equation}
The first term on the right hand side will be treated differently from
the last two terms, which will be considered as source terms. Using
the formula
\begin{equation*}
  |z_2|^2z_2 - |z_1|^2z_1 = |z_2|^2(z_2-z_1) +z_1\RE
  (z_2-z_1)(\overline z_2+\overline z_1),
\end{equation*}
with $z_2 = \F(Mv)$ and $z_1=\eps\F(Mv_0)$, as well as
\eqref{eq:recoll-DA}, yielding
\begin{equation}\label{eq:redress-DA}
  \F\(M(v-\eps v_0)\) = \F M \F^{-1}\(  \overline B \( \widehat w -\eps
  \widehat v_0-\eps^3\widehat \nu_2\) +  \eps (\overline B -1)\widehat
  v_0 +\eps^3 \overline B \widehat \nu_2\),
\end{equation}
we write
\begin{equation*}
  \|I_1^\eps-\eps^3J_1\|_{L^2}\le \|G_1\|_{L^2} +
  \|S_1\|_{L^2}+\|S_2\|_{L^2}, 
\end{equation*}
where, simplifying the algebraic structure for the sake of presentation,
\begin{align*}
  G_1& = \F\(M^{-1}-1\)\F^{-1}\( \(|\F(M v)|^2 + \eps^2|\F(Mv_0)|^2\) \F
       M \F^{-1} \( \overline B \widehat \rho \)\),\\
  S_1& = \eps \F\(M^{-1}-1\)\F^{-1}\( \(|\F(M v)|^2 +
       \eps^2|\F(Mv_0)|^2\) \F 
       M \F^{-1} \( (\overline B-1) \widehat v_0 \)\),\\
  S_2 & =  \eps^3 \F\(M^{-1}-1\)\F^{-1}\( \(|\F(M v)|^2 +
       \eps^2|\F(Mv_0)|^2\) \F 
       M \F^{-1} \( \overline B \widehat \nu_2 \)\).
\end{align*}
We have left out the time variable in the expression
  of $M$ in order to lighten the notation, and do so below for the
  same reason.
The Gronwall term $G_1$ is estimated by merely using the boundedness
on $L^2$ of $ \F\(M^{-1}-1\)\F^{-1}$, and the unitarity of $\F M
\F^{-1}$, and we write
\begin{equation*}
  \|G_1\|_{L^2} \lesssim \(\|\F(Mv)\|_{L^\infty}^2 +\eps^2
  \|\F(Mv_0)\|_{L^\infty}^2  \) \|\widehat \rho\|_{L^2}\lesssim \eps^2
  \|\widehat \rho\|_{L^2}, 
\end{equation*}
by using \eqref{eq:control-hat-v}. For the source term $S_1$ and
$S_2$, we invoke the same arguments as in the proof of
Lemma~\ref{lem:R}, and the fact that $\F 
       M \F^{-1}$ is unitary on $H^s$ for any $s\ge 0$:
\begin{align*}
  \|S_1\|_{L^2}&\lesssim \frac{\eps}{\sqrt t} \left\| \(|\F(M v)|^2 +
       \eps^2|\F(Mv_0)|^2\) \F 
       M \F^{-1} \( (\overline B-1) \widehat v_0 \)\right\|_{H^1}\\
& \lesssim \frac{\eps}{\sqrt t}\(\|v\|_{\F (H^1)}^2 + \eps^2
  \|v_0\|_{\F (H^1)}^2 \) \|(\overline B-1)\widehat v_0\|_{H^1}\\
& \lesssim \frac{\eps^3}{\sqrt t}t^{2C\eps^2} \|(\overline B-1)\widehat
 v_0\|_{H^1}, 
\end{align*}
where we have used \eqref{eq:unif-u} for the last inequality. The last
term is controlled by
\begin{equation*}
  \|(\overline B-1)\widehat
 v_0\|_{H^1}\le \|\overline B-1\|_{L^\infty} \|v_0\|_\Sigma + \|\widehat
 v_0\d_x \overline B \|_{L^2}\lesssim \eps^2\log t +  \|\widehat
 v_0\d_x \overline B \|_{L^2}, 
\end{equation*}
in view of \eqref{eq:B-1}. Writing
\begin{equation*}
  \widehat
 v_0\d_x \overline B  = -2i\lambda\widehat v_0 \int_1^t
 \RE\(\overline{\widehat v}\d_x \widehat v\)\frac{d\tau}{\tau},
\end{equation*}
\eqref{eq:unif-u} now yields
\begin{equation*}
 \|\widehat
 v_0\d_x \overline B \|_{L^2}\lesssim \|v_0\|_{L^2}\int_1^t  \|
 \widehat v(\tau)\|_{L^\infty} \|\d_x \widehat
 v(\tau)\|_{L^2}\frac{d\tau}{\tau}\lesssim \eps^2 \int_1^t
 \frac{d\tau}{\tau^{1-C\eps^2}} \lesssim \eps^2 t^{C\eps^2}.
\end{equation*}
The estimate for $S_2$ is rather similar: 
\begin{align*}
  \|S_2\|_{L^2} & \lesssim \frac{\eps^3}{t^{1/2}}\left\| \(|\F(M v)|^2 +
       \eps^2|\F(Mv_0)|^2\) \F  
       M \F^{-1} \( \overline B \widehat \nu_2 \)\right\|_{H^1}\\
&\lesssim \frac{\eps^3}{t^{1/2}}\(\|\F(M v)\|_{H^1}^2 +
       \eps^2\|\F(Mv_0)\|_{H^1}^2\)\|\overline B \widehat \nu_2
 \|_{H^1}
    \lesssim \frac{\eps^5}{t^{1/2-2C\eps^2}}\|\overline B \widehat \nu_2
   \|_{H^1} . 
\end{align*}
The last term is controlled via Leibniz formula, invoking
Lemma~\ref{lem:nu2}, 
\begin{equation*}
\|\overline B \widehat \nu_2
   \|_{H^1}\le \|\widehat \nu_2   \|_{H^1}   + \|\widehat \nu_2\d_x
   \overline B\|_{L^2}\lesssim \log t + \eps^2 t^{C\eps^2},
\end{equation*}
and the source terms are estimated by
\begin{equation*}
  \|S_1\|_{L^2} + \|S_2\|_{L^2} \lesssim \eps^5\frac{\log
    t}{t^{1/2-3C\eps^2}}. 
\end{equation*}
We now turn to the $L^2$ estimate of $I_2^\eps-\eps^3 J_2$, and proceed
along the same spirit. 
\begin{equation*}
  I_2^\eps-\eps^3 J_2 = \left|\widehat{M(t)v}\right|^2
       \widehat{M(t)v} - |\widehat v|^2\widehat v-\eps^3
       \left|\widehat{M(t)v_0}\right|^2 
       \widehat{M(t)v_0} +\eps^3 |\widehat v_0 |^2\widehat v_0 .
\end{equation*}
We distinguish $\left|\widehat{M(t)v}\right|^2
   \widehat{M(t)v} -\eps^3 \left|\widehat{M(t)v_0}\right|^2 
 \widehat{M(t)v_0} $ and $|\widehat v|^2\widehat v-\eps^3
 |\widehat v_0 |^2\widehat v_0$.
 Discarding the precise algebraic structure like before,
 \begin{equation*}
     \left|\widehat{M(t)v}\right|^2
   \widehat{M(t)v} -\eps^3 \left|\widehat{M(t)v_0}\right|^2 
 \widehat{M(t)v_0} \approx \(    \left|\widehat{M(t)v}\right|^2+
   \eps^2  \left|\widehat{M(t)v_0}\right|^2 \) \F M(v-\eps v_0). 
  \end{equation*}
The last factor is again rewritten thanks to \eqref{eq:redress-DA},
and we estimate $I_2^\eps-\eps^3 J_2$ as
\begin{equation*}
  \|I_2^\eps-\eps^3 J_2\|_{L^2}\le \|G_2\|_{L^2}+\|S_3\|_{L^2} + \|S_4\|_{L^2},
\end{equation*}
where
\begin{align*}
  G_2 &= \(    \left|\widehat{M(t)v}\right|^2+
   \eps^2  \left|\widehat{M(t)v_0}\right|^2 \) \F M\F^{-1}\(\overline
        B \( \widehat w -\eps
        \widehat v_0-\eps^3\widehat \nu_2\) \)\\
  &\quad- \(    \left|\widehat{v}\right|^2+
   \eps^2  \left|\widehat{v_0}\right|^2 \) \overline
        B \( \widehat w -\eps
    \widehat v_0-\eps^3\widehat \nu_2\) ,\\
    S_3& = \eps \(    \left|\widehat{M(t)v}\right|^2+
   \eps^2  \left|\widehat{M(t)v_0}\right|^2 \) \F M\F^{-1}\(\(\overline
         B -1\) \widehat v_0\)\\
   &\quad -\eps \(    \left|\widehat{v}\right|^2+
   \eps^2  \left|\widehat{v_0}\right|^2 \) \(\overline
     B -1\) \widehat v_0,\\
    S_4& = \eps^3 \(    \left|\widehat{M(t)v}\right|^2+
   \eps^2  \left|\widehat{M(t)v_0}\right|^2 \) \F M\F^{-1}\widehat \nu_2 -\eps^3 \(    \left|\widehat{v}\right|^2+
   \eps^2  \left|\widehat{v_0}\right|^2 \)  \widehat \nu_2. 
\end{align*}
For the Gronwall term $G_2$, we proceed like for $G_1$, and write
\begin{equation*}
  \|G_2\|_{L^2}\lesssim \eps^2\| \widehat w -\eps
    \widehat v_0-\eps^3\widehat \nu_2\|_{L^2}=\eps^2\|\widehat \rho\|_{L^2}.
\end{equation*}
The source terms $S_3$ and $S_4$ are readily of size $\O(\eps^5)$; we
recover some decay in time by making the quantity $M-1$ appear
systematically. In the case of $S_3$, we write
\begin{align*}
   S_3& = \eps \(    \left|\widehat{M(t)v}\right|^2+
   \eps^2  \left|\widehat{M(t)v_0}\right|^2 \) \F M\F^{-1}\(\(\overline
         B -1\) \widehat v_0\)\\
   &\quad -\eps \(    \left|\widehat{v}\right|^2+
   \eps^2  \left|\widehat{v_0}\right|^2 \) \(\overline
     B -1\) \widehat v_0
   \pm \eps \(    \left|\widehat{M(t)v}\right|^2+
   \eps^2  \left|\widehat{M(t)v_0}\right|^2 \) \(\overline
     B -1\) \widehat v_0,
\end{align*}
and estimate as follows:
\begin{align*}
  \|S_3\|_{L^2} & \lesssim \eps^3 \| (M-1)\F^{-1}\(\(\overline
                  B -1\) \widehat v_0\)\|_{L^2} \\
  &\quad+ \|\(\overline
         B -1\) \widehat v_0\|_{L^\infty} \|(M-1)\widehat v\|_{L^2}\(
    \|\widehat{Mv}\|_{L^\infty} +  \|\widehat{v}\|_{L^\infty}\)\\
  & \lesssim \frac{\eps^3}{\sqrt t} \| x\F^{-1}\(\(\overline
    B -1\) \widehat v_0\)\|_{L^2}
    + \eps^5 \frac{\log t}{\sqrt t} \times\|x\widehat v\|_{L^2},
\end{align*}
where we have used \eqref{eq:control-hat-v} and \eqref{eq:B-1}. We
have already estimated the $H^1$-norm of $\(\overline
B -1\) \widehat v_0$,
\begin{equation*}
  \|\(\overline B -1\) \widehat v_0\|_{H^1}\lesssim \eps^2
 \(\log t + t^{C\eps^2}\),
\end{equation*}
and therefore
\begin{equation*}
  \|S_3\|_{L^2} \lesssim \eps^5\frac{\log t}{t^{1/2- C\eps^2}}.  
\end{equation*}
The term $S_4$ is controlled similarly, by using the same ideas as
above, and we come up with:
\begin{equation*}
  \frac{d}{dt}\|\widehat \rho\|_{L^2}\lesssim
  \frac{\eps^2}{t}\|\widehat \rho\|_{L^2} + \eps^5 \frac{\log t}{t^{3/2- 3C\eps^2}}.
\end{equation*}
Gronwall lemma then implies, provided that $\eps>0$ is
sufficiently small:
  \begin{proposition}\label{prop:approx-nu2}
    Suppose that $u^\eps$ solves \eqref{eq:NLS}, with $u^\eps_{\mid
      t=0}=u_0^\eps\in \Sigma$ such that 
    \begin{equation*}
      u_0^\eps = \eps v_0+\eps^3 w_2 + \O\(\eps^{5-\eta}\),
    \end{equation*}
    for $v_0,w_2\in \Sigma$, and some $0<\eta<2$. Then we have
    \begin{equation*}
  \sup_{t\ge 1} \|\widehat w (t)-\eps
    \widehat v_0-\eps^3\widehat \nu_2(t)\|_{L^2}\lesssim \eps^{5-\eta},
  \end{equation*}
  where $\nu_2$ is defined in Lemma~\ref{lem:nu2}. 
  \end{proposition}
  In particular, letting $t$ go to infinity, we infer
  \begin{equation*}
    W =\eps \widehat v_0+\eps^3\widehat \nu_2^\infty +
    \O\(\eps^{5-\eta}\)\quad \text{in }L^2. 
  \end{equation*}
As $u_+^\eps = \F^{-1}\(W e^{-i\Phi}\)$, and we have seen at the end
of Section~\ref{sec:complete} that $\|\Phi\|_{L^\infty}=\O(\eps^4)$,
we infer
\begin{equation*}
  u_+^\eps  = \eps v_0 + \eps^3\nu_2 +  \O\(\eps^{5-\eta}\)\quad \text{in }L^2,
\end{equation*}
thus completing the proof of Theorem~\ref{theo:main}.

\subsection*{Acknowledgments.} The author wishes to thank L\'eo
Bigorgne for discussions on this topic, as well as Satoshi
Masaki and Pavel Naumkin for useful explanations on
Theorem~\ref{theo:complete}.  The author is grateful to the anonymous
referee for their careful reading, and for pointing out some flaw in
the initial submission.

\bibliographystyle{abbrv}
\bibliography{biblio}

\begin{thebibliography}{10}

\bibitem{Barab}
J.~E. Barab.
\newblock Nonexistence of asymptotically free solutions for nonlinear
  {S}chr\"odinger equation.
\newblock {\em J. Math. Phys.}, 25:3270--3273, 1984.

\bibitem{CaCMP}
R.~Carles.
\newblock Geometric optics and long range scattering for one-dimensional
  nonlinear {S}chr\"odinger equations.
\newblock {\em Comm. Math. Phys.}, 220(1):41--67, 2001.

\bibitem{CaGa09}
R.~Carles and I.~Gallagher.
\newblock Analyticity of the scattering operator for semilinear dispersive
  equations.
\newblock {\em Comm. Math. Phys.}, 286(3):1181--1209, 2009.

\bibitem{MR3724144}
T.~Cazenave and I.~Naumkin.
\newblock Modified scattering for the critical nonlinear {S}chr\"{o}dinger
  equation.
\newblock {\em J. Funct. Anal.}, 274(2):402--432, 2018.

\bibitem{ChenMurphy}
G.~Chen and J.~Murphy.
\newblock Recovery of the nonlinearity from the modified scattering map.
\newblock {\em Int. Math. Res. Not. IMRN}, 2024.
\newblock To appear. Archived at \url{https://arxiv.org/abs/2304.01455}.

\bibitem{DG}
J.~Derezi\'nski and C.~G\'erard.
\newblock {\em Scattering theory of quantum and classical {N}-particle
  systems}.
\newblock Texts and Monographs in Physics, Springer Verlag, Berlin Heidelberg,
  1997.

\bibitem{Ginibre}
J.~Ginibre.
\newblock An introduction to nonlinear {S}chr\"odinger equations.
\newblock In R.~Agemi, Y.~Giga, and T.~Ozawa, editors, {\em Nonlinear waves
  (Sapporo, 1995)}, GAKUTO International Series, Math. Sciences and Appl.,
  pages 85--133. Gakk\={o}tosho, Tokyo, 1997.

\bibitem{MR1855975}
J.~Ginibre and G.~Velo.
\newblock Long range scattering and modified wave operators for some {H}artree
  type equations. {III}. {G}evrey spaces and low dimensions.
\newblock {\em J. Differential Equations}, 175(2):415--501, 2001.

\bibitem{MR2342882}
J.~Ginibre and G.~Velo.
\newblock Long range scattering and modified wave operators for the
  {M}axwell-{S}chr\"{o}dinger system. {II}. {T}he general case.
\newblock {\em Ann. Henri Poincar\'{e}}, 8(5):917--994, 2007.

\bibitem{MR2474176}
J.~Ginibre and G.~Velo.
\newblock Long range scattering for the {M}axwell-{S}chr\"{o}dinger system with
  arbitrarily large asymptotic data.
\newblock {\em Hokkaido Math. J.}, 37(4):795--811, 2008.

\bibitem{MR2853555}
J.~Ginibre and G.~Velo.
\newblock Long range scattering for the wave-{S}chr\"{o}dinger system
  revisited.
\newblock {\em J. Differential Equations}, 252(2):1642--1667, 2012.

\bibitem{MR3192651}
J.~Ginibre and G.~Velo.
\newblock Modified wave operators without loss of regularity for some
  long-range {H}artree equations: {I}.
\newblock {\em Ann. Henri Poincar\'{e}}, 15(5):829--862, 2014.

\bibitem{MR3359525}
J.~Ginibre and G.~Velo.
\newblock Modified wave operators without loss of regularity for some long
  range {H}artree equations. {II}.
\newblock {\em Commun. Pure Appl. Anal.}, 14(4):1357--1376, 2015.

\bibitem{MR3144794}
Z.~Guo, N.~Hayashi, Y.~Lin, and P.~I. Naumkin.
\newblock Modified scattering operator for the derivative nonlinear
  {S}chr\"{o}dinger equation.
\newblock {\em SIAM J. Math. Anal.}, 45(6):3854--3871, 2013.

\bibitem{HN98}
N.~Hayashi and P.~Naumkin.
\newblock Asymptotics for large time of solutions to the nonlinear
  {S}chr\"odinger and {H}artree equations.
\newblock {\em Amer. J. Math.}, 120(2):369--389, 1998.

\bibitem{HN06}
N.~Hayashi and P.~Naumkin.
\newblock Domain and range of the modified wave operator for {S}chr\"odinger
  equations with a critical nonlinearity.
\newblock {\em Comm. Math. Phys.}, 267(2):477--492, 2006.

\bibitem{MR1618664}
N.~Hayashi and P.~I. Naumkin.
\newblock Asymptotic behavior in time of solutions to the derivative nonlinear
  {S}chr\"{o}dinger equation.
\newblock {\em Ann. Inst. H. Poincar\'{e} Phys. Th\'{e}or.}, 68(2):159--177,
  1998.

\bibitem{MR2047418}
N.~Hayashi, P.~I. Naumkin, A.~Shimomura, and S.~Tonegawa.
\newblock Modified wave operators for nonlinear {S}chr\"{o}dinger equations in
  one and two dimensions.
\newblock {\em Electron. J. Differential Equations}, pages No. 62, 16, 2004.

\bibitem{MR2864547}
N.~Hayashi, H.~Wang, and P.~I. Naumkin.
\newblock Modified wave operators for nonlinear {S}chr\"{o}dinger equations in
  lower order {S}obolev spaces.
\newblock {\em J. Hyperbolic Differ. Equ.}, 8(4):759--775, 2011.

\bibitem{KatoPusateri2011}
J.~Kato and F.~Pusateri.
\newblock A new proof of long-range scattering for critical nonlinear
  {Schr{\"o}dinger} equations.
\newblock {\em Differ. Integral Equ.}, 24(9-10):923--940, 2011.

\bibitem{MR4576319}
R.~Killip, J.~Murphy, and M.~Visan.
\newblock The scattering map determines the nonlinearity.
\newblock {\em Proc. Amer. Math. Soc.}, 151(6):2543--2557, 2023.

\bibitem{KiMuVi-p}
R.~Killip, J.~Murphy, and M.~Visan.
\newblock Determination of {Schr{\"o}dinger} nonlinearities from the scattering
  map.
\newblock Preprint, archived at \url{https://arxiv.org/abs/2402.03218}, 2024.

\bibitem{MR1913680}
N.~Kita and T.~Wada.
\newblock Sharp asymptotic behavior of solutions to nonlinear {S}chr\"{o}dinger
  equations in one space dimension.
\newblock {\em Funkcial. Ekvac.}, 45(1):53--69, 2002.

\bibitem{LindbladMurphy2006}
H.~Lindblad and A.~Soffer.
\newblock Scattering and small data completeness for the critical nonlinear
  {S}chr\"{o}dinger equation.
\newblock {\em Nonlinearity}, 19(2):345--353, 2006.

\bibitem{MaMi18}
S.~Masaki and H.~Miyazaki.
\newblock Long range scattering for nonlinear {S}chr\"{o}dinger equations with
  critical homogeneous nonlinearity.
\newblock {\em SIAM J. Math. Anal.}, 50(3):3251--3270, 2018.

\bibitem{MaMi19}
S.~Masaki, H.~Miyazaki, and K.~Uriya.
\newblock Long-range scattering for nonlinear {S}chr\"{o}dinger equations with
  critical homogeneous nonlinearity in three space dimensions.
\newblock {\em Trans. Amer. Math. Soc.}, 371(11):7925--7947, 2019.

\bibitem{MTT03}
K.~Moriyama, S.~Tonegawa, and Y.~Tsutsumi.
\newblock Wave operators for the nonlinear {S}chr\"{o}dinger equation with a
  nonlinearity of low degree in one or two space dimensions.
\newblock {\em Commun. Contemp. Math.}, 5(6):983--996, 2003.

\bibitem{Ozawa91}
T.~Ozawa.
\newblock Long range scattering for nonlinear {S}chr\"odinger equations in one
  space dimension.
\newblock {\em Comm. Math. Phys.}, 139:479--493, 1991.

\bibitem{Wada2000}
T.~Wada.
\newblock A remark on long-range scattering for the {H}artree type equation.
\newblock {\em Kyushu J. Math.}, 54(1):171--179, 2000.

\end{thebibliography}

\end{document}